\documentclass{amsart}
\usepackage{pgf,pgfarrows,pgfnodes,pgfautomata,pgfheaps,pgfshade,hyperref, amssymb}
\usepackage{enumerate}
\usepackage[capitalize]{cleveref}
\usepackage{mathtools,tikz}
\usepackage{ulem}

\newtheorem{theo}{Theorem}[section]
\newtheorem{prop}[theo]{Proposition}

\newtheorem{lem}[theo]{Lemma}
\newtheorem{cor}[theo]{Corollary}

\newtheorem{example}[theo]{Example}

\newtheorem{defi}[theo]{Definition}

\newtheorem{ques}[theo]{Question}
\newtheorem*{ques*}{Question}

\def\R{{\mathbb R}}
\def\Z{{\mathbb Z}}

\def\N{{\mathbb N}}

\def\S{{\mathbb S}}

\def\cH{{\mathcal H}}

\def\cS{{\mathcal S}}
\def\cX{{\mathcal X}}
\def\ie{{i.e.}}  

\def \RP {{\bf RP}}

\def \ND {\operatorname{ND}_{\ell_2}}
\def \NE {\operatorname{NE}}
\def\dist{{\rm dist}}

\title[A geometric framework for asymptoticity and expansivity]{A geometric framework for asymptoticity and expansivity in topological dynamics}

\author{Sebasti\'an Donoso}
\address{Departamento de Ingenier\'{\i}a
	Matem\'atica and Centro de Modelamiento Ma\-te\-m\'a\-ti\-co, Universidad de Chile and IRL-CNRS 2807, Beauchef 851, Santiago,
	Chile.} 
\email{sdonoso@dim.uchile.cl}

\author{Alejandro Maass}
\address{Departamento de Ingenier\'{\i}a
	Matem\'atica and Centro de Modelamiento Ma\-te\-m\'a\-ti\-co, Universidad de Chile and IRL-CNRS 2807, Beauchef 851, Santiago,
	Chile.}
\email{amaass@dim.uchile.cl}

\author{Samuel Petite}
\address{Laboratoire Ami\'enois
	de Math\'ematiques Fondamentales et Appliqu\'ees, CNRS-UMR 7352, Universit\'{e} de Picardie Jules Verne, 33 rue Saint Leu, 80039   Amiens cedex 1,
	France.} \email{samuel.petite@u-picardie.fr}

\thanks{The first and second authors were partially funded by Centro de Modelamiento Matemático (CMM), ACE210010 and FB210005, BASAL funds for centers of excellence from ANID-Chile. The third author was partially funded by the ANR Project IZES  ANR-22-CE40-0011. All authors are part of the ECOS-ANID grant C21E04 (ECOS210033). The first author was partially funded by ANID/Fondecyt/1200897}

\subjclass[2020]{Primary: 37B05; Secondary: 54H15, 20F65, 37B10, 37B20}

\keywords{asymptoticity, expansivity, horoballs}

\begin{document}
\date{\today}
\maketitle

\begin{abstract}
In this paper we develop a geometric framework to address asymptoticity and nonexpansivity in topological dynamics when the acting group is second countable and locally compact. As an application, we show extensions of Schwartzman's theorem in this context. Also, we get new results when the acting groups is ${\mathbb Z}^d$: any half-space of $\R^d$ contains a vector defining a (oriented) nonexpansive direction in the sense of Boyle and Lind. Finally, we deduce rigidity properties of distal Cantor systems.
\end{abstract}

\section{Introduction}
Let $T\colon X\to X$ be a self-homeomorphism of an infinite compact metric space $X$. S. Schwartzman proved in \cite{Sch} the following celebrated result: given $\epsilon >0$, there exist different points $x,y \in X$ such that for all $n\geq 0$ one has 
$${\rm dist}(T^nx,T^ny)\leq \epsilon,$$
where ${\rm dist}$ is the distance in $X$. That is, there exist positive $\epsilon$-asymptotic pairs. Considering $T^{-1}$ in place of $T$, the same result implies that there exist negative $\epsilon$-asymptotic pairs \footnote{This means that there are two different points that are $\epsilon$-close for all $n\leq 0$.}. Another way to state the same result is by saying that in an infinite compact metric space it is impossible to construct positively expansive homeomorphisms. This result has been recovered independently and in different flavours by several authors (see for instance \cite{AKM,BO,BD, GottschalkHedlund, KeynesRobertson, King, MS, RichesonWiseman} and the survey \cite{KeaneCoven}).

Motivated by the study of subactions of higher rank abelian actions, M. Boyle and D. Lind obtained in \cite{BD} a major result in the spirit of Schwarztman's theorem. They introduced the notion of {\em nonexpansivity} along a subspace of $\R^d$, by considering only the action of elements of $\Z^d$ that lie within a given bounded distance of the subspace. They showed (see also \cite{ELMW}) that for an infinite compact metric space $X$, there exist a half-space $H \subseteq \R^d$ and different points $x,y$ in $X$ that stay arbitrarily close along their $H\cap \Z^d$-orbit (see Section \ref{sec:closedness}). Such half-space is called {\em nonexpansive} since the action of its border is nonexpansive.

Let us point out that unlike Schwartzman's result, where the two half-spaces of $\Z$ are nonexpansive, here only one nonexpansive half-space is guaranteed to exist. Actually, not any half-space is nonexpansive for an arbitrary $\Z^2$ system. Explicit examples like Ledrappier's $\Z^2$ subshift illustrate that the set of such half-spaces can even be finite. A motivation for studying the set of nonexpansive half-spaces is that many dynamical properties of the subactions vary nicely  along  the connected components of the set of expansive directions and a bifurcation phenomenon may occur when passing from one component to another. 

To complete their result, Boyle and Lind, and  later M. Hochman, realized any nonempty closed set of half-spaces as the set of nonexpansive ones of some $\Z^2$-action. In particular, Hochman succeeded in realizing the challenging case where this set consists of a single half-space whose border line has an  irrational slope \cite{Hochman:2011}. 

The notions of nonexpansive half-spaces and asymptotic points have turned out to be the key notions to tackle various problems. For instance, in dynamics, they appear as fundamental objects for studying expansive maps (e.g., \cite{M}), their centralizer and cellular automata (e.g., \cite{Nasu,DonosoDurandMaassPetite:2016}), or to study topological joinings (e.g., \cite{King}). For $\Z^d$-actions, they are main notions to explore directional invariants like directional entropy (e.g., \cite{Park}). They also appear in combinatorics when addressing the long standing Nivat conjecture (e.g., \cite{CyrKra:2015}). Despite the breadth of possible applications, no similar concept has been proposed for general group actions. 

The definition and the proofs in \cite{BD} seem to rely heavily on the linear structure of $\R^d$, so extensions of these results to general group actions were open for a long time. In this work, we develop a framework using basic elements of geometric group theory that allows us to give a  definition of nonexpansiveness for actions of  countable groups and more generally for second countable and locally compact topological groups.  
The notion of half-space is replaced by that of horoball (see \cref{sec:horoball} for the precise definition). Horoballs are sublevel sets  of horofunctions, which are the analogous to linear forms on general metric spaces. They depend on the choice of a proper right-invariant distance on the acting group. We warn the reader that, depending on the group and the metric, horoballs can contain arbitrary large balls, like in the case of finitely generated groups, or can even be empty, like in the case of infinite sum of finite groups (see \cref{lem:BigHoroball} and the discussion around it). Anyway, using this framework, we show the following extension of Boyle and Lind's results, which is meaningful for a large class of groups including the finitely generated ones.

\begin{theo}\label{thm_1_intro} Let $ (X,T,G)$ be an infinite topological dynamical system and  $G $ be an infinite  group with a proper right-invariant distance. Then, for all $\epsilon >0$ there exist different points $x, y \in X$ and a horoball $H$ such that $\dist(T_gx,T_gy)\leq \epsilon$ for all $g\in H$.
\end{theo}

Actually, \cref{thm_1_intro} is derived from the following more general theorem, which we believe is of independent interest and might be useful for other problems in topological dynamics. 
\begin{theo}[Robinson Crusoe theorem]\label{theo1:SchwartzmannVersionAbstraite} Let $ (X,T,G)$ be a topological dynamical system and  $G $ be an infinite group with a proper right-invariant distance. Let $O \subsetneq X$ be an open, not closed and $G$-invariant ($T_g(O) = O, \forall g \in G$) subset of $X$. Then, for any neighborhood $U$ of the boundary of $O$ there exists a horoball $H$ in $G$ such that 
	$$O \cap \bigcap_{g \in H}T_g^{-1}(U)\neq \emptyset.$$
\end{theo}

 We name this theorem Robinson Crusoe because of the following interpretation. Assume that Robinson Crusoe is isolated on an island and his movement is deterministic, \ie, given by a $\R$-flow. The island defines a $\R$-invariant open set $O$ in the phase space. The theorem claims that there exists a trajectory of Robinson Crusoe that stays on the beach (the border of the island) for all positive (or negative) times. This can be seen as a containment result: $O$ is not repulsed by its border for positive or negative times.    

From \cref{theo1:SchwartzmannVersionAbstraite} we deduce in \cref{theo:factors_general1} the existence of $\epsilon$-asymptotic pairs relatively to a factor map, and from this \cref{thm_1_intro} follows immediately. When $G=\Z^d$ is endowed with the Euclidean metric, we recover the theorem of Boyle and Lind on the existence of nonexpansive half-spaces. 

Pursuing extensions of Schwartzman's result, we also provide a directed version of Robinson Crusoe theorem where the set of horoballs can be restricted to a smaller family, as long as the action satisfies a non-repulsive condition along a semi-group (\cref{thm:directed_RC}). Although this technical condition restricts the actions, the family of systems that satisfy it is rather broad. We deduce in \cref{theo:factors_general} a directed version of asymptotic pairs for infinite groups admitting a proper right-invariant distance such that the inversion map $g\mapsto g^{-1}$ is an isometry. As a consequence, for $G=\Z^d$ we obtain the following new restriction on the nonexpansive half-spaces of a given action (see \cref{cor:half_space_Z2}).

\begin{theo} \label{Thm:intro2}
Let $(X,T,\Z^d)$ be an infinite topological dynamical system. 
Then, the intersection of all its nonexpansive open half-spaces is empty.
\end{theo}

Whereas \cref{thm_1_intro}  states that there exists at least  one nonexpansive half-space for $\Z^d$-actions, \cref{Thm:intro2} establishes that the half-space cannot be unique. In particular, when $d=1$, these recover Schwartzman's result. Thus, Theorems \ref{thm_1_intro} and \ref{Thm:intro2} constitute the first known common generalization of Schwartzman's and Boyle and Lind's results.

Finally, we apply  Robinson Crusoe theorem  to Cantor dynamics to get restrictions on  distal systems and distal factors for a class of groups containing finitely generated groups: they have to be equicontinuous (\cref{cor:factor_distals_are_equicontinuous}). This recovers results in \cite{AGW} and in \cite{MS} for subshifts.

\subsection*{Organization of the paper} In \cref{sec:horoball} we introduce the basic tools we need from geometric group theory.  \cref{sec:Robinson_Crusoe} is devoted to proving the main results of the paper: the { Robinson Crusoe theorem} (\cref{theo:SchwartzmannVersionAbstraite}) and the { Directed Robinson Crusoe theorem} (\cref{thm:directed_RC}). Then, in \cref{sec:asymptotics} we derive applications of these theorems to the notions of nonexpansiveness in topological dynamics. In \cref{sec:closedness} we focus on $\Z^d$-actions, and in particular we show how to obtain \cref{Thm:intro2} (\cref{cor:half_space_Z2}). In this section we also discuss the dependency of our main theorems on the metric chosen in $\Z^d$ and show interesting phenomena. The last section, \cref{sec:Cantor_dynamics}, is devoted to applications in Cantor dynamics that are of independent interest. 

\section*{Acknowledgement}
We are grateful to Benjamin Weiss and Xiangdong Ye for pointing out an error in an earlier version of this paper, namely that \cref{cor:NoDistalCantorAction} does not hold for any countable group. 

\section{Preliminaries}\label{sec:horoball}
\subsection{Basics on groups geometry}
In this section, we recall some basic facts on horoballs on groups.  
Given a group $G$ we denote by $1_G$ its neutral element. In what follows we will assume that the group $G$ is topological, second countable and locally compact. Equivalently  (see \cite{Struble}), it admits a 
distance $d \colon G \times G \to \R$ which is 
\begin{itemize}
\item right-invariant, \ie, $d(gf, hf) = d(g,h)$ for all $g,h,f \in G$;
\item and proper, \ie, every closed ball is compact. 
\end{itemize}
We remark that even if we will work with a distance, all the results presented in this paper could be extended to any continuous right-invariant and proper semi-distance. 

Any countable group satisfies these hypotheses for an integer valued  distance. For instance, when $G$ is a group generated by a finite set $\mathcal S$, a proper right-invariant distance is given by  the $\ell_1$ distance: $d_{\ell_1}(g,h) = \| gh^{-1} \|_1  = \inf \{n\in \N: gh^{-1} = s_{1}\cdots s_{n} \text{ for } s_{1},\ldots,s_n \in {\mathcal S} \cup {\mathcal S}^{-1} \}$. In the general case ({e.g.}, when  $G$ is not finitely generated), the classical Higman-Neumann-Neumann Theorem ensures that $G$ is a subgroup of a finitely generated group so that the former distance induces a proper right-invariant distance on $G$. A simple and constructive way is the following. Consider a sequence of  finite subsets $(B_n)_{n \ge 0}$ of $G$ such that \begin{itemize}
    \item $B_0 = \{1_G \}$;
    \item each $B_n$ is symmetric, \ie, $B_n = B_n^{-1}$;
    \item $B_n B_m \subseteq B_{n+m}$ for any $n,m\ge 0$;
    \item $G = \bigcup_{n\ge 0} B_n$.
\end{itemize}
Then, $\tilde{d} (g,h) = \inf \{ n \in \N: gh^{-1} \in B_n  \}$ defines a proper right-invariant distance and each $B_n$ is the open ball of radius $n$ centered at the neutral element $1_G$ for this distance.

Generalizing the ideas of H. Busemann, M. Gromov defines a compactification of the group $G$ with an embedding  map $b$. Denoting by $C(G)$ the collection of continuous real functions on $G$, this embedding is given by, 
\begin{eqnarray*}
b\colon G & \to & C(G)\\
g & \mapsto& b_{g} \colon x \mapsto d(g,x) -d(g,1_G). 
\end{eqnarray*}
The triangle inequality implies that all the maps $b_{g}$ are 1-Lipschitz. Moreover, by construction, we have  $b_{g}(1_G) =0$. It follows from  Arzel\`a-Ascoli's theorem and a diagonal argument that $b(G)$ is a relatively  compact set in $C(G)$ for the compact-open topology. It is straightforward to check that the map $b$ is an injection. 
   
The {\em border}  of $G$, denoted by $\partial G$, is the set 
$$ \partial G = \overline{b(G)} \setminus b(G),$$
where the closure is considered for the compact-open topology (i.e., here it corresponds to the uniform convergence on compact sets, which is  first countable). This set is not empty if $G$ is unbounded. 
Any function   $j \in \partial G$ is called an {\em horofunction} (or a {\em Busemann cocycle}). Moreover, there exists a sequence $(g_{n})_{n\in \N} \subseteq G$ going to infinity \footnote{This means that for any $m$, $g_n \notin B_{m}(1_G)$ for all large enough $n$, i.e., $g_n$ eventually leaves any ball.}  such that $\displaystyle j = \lim_{n\to\infty} b_{g_{n}}$. 
We call {\em horoball} any subset of $G$ of the form
$$H= \{x \in G: \ j(x)< 0\} \textrm{ for some element } j \in \partial G.$$
We denote by $\mathcal H$ the collection of all horoballs.  Notice that each horoball can be seen as the limit of sets of the form $\{ g \in G : b_{g_n}(g) <0 \}$, \ie, open balls centered at $g_n$ with radius $d(g_n, 1_G)$. 
\begin{figure}[t]
\begin{tikzpicture}
\draw [gray!20, fill=gray!20] (0,-2) rectangle (4,2);
\draw (0.4,0) circle (0.4) ;
\draw (0.925,0) circle (0.925) ;
\draw (2,0) circle (2) ;
\draw (0.7,2) arc (140:224:3);
\draw [very thick] (0,-2) -- (0,2);
\draw  (0,0) node {$\times$};
\draw  (0,0) node[below left] {$O$};
\draw  (0.4,0) node {$\times$};
\draw  (0.4,0) node[below  ] {$g_1$};
\draw  (0.925,0) node {$\times$};
\draw  (0.925,0) node[below right] {$g_2$};
\draw  (2,0) node {$\times$};
\draw  (2,0) node[below right] {$g_3$};
\end{tikzpicture}
\caption{A $\ell_2$ horoball (in gray) for the group $\Z^2$ obtained as a limit of balls.}\label{fig:L2horoball}
\end{figure}
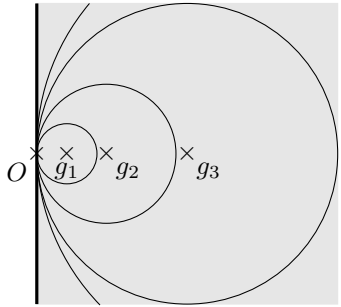

The notion of horoball depends on the distance $d$. 
For instance, for the $\ell_2$ norm on $\Z^d$, a standard computation shows that any horofunction  which is the limit of functions $b_{g_n}$, with $g_n$ going to infinity, is of the form $\langle v , \cdot \rangle$ for some vector $v \in \R^d$ that is the accumulation point of the vectors $g_n/\|g_n\|_2$. Conversely, any unit vector $v$ can be obtained in this way. Hence, a horoball is an open half-space delimited by a hyperplane, as illustrated in \cref{fig:L2horoball}.
In the sequel we will identify the $\ell_2$ horoball $\{ x \in \Z^d: \langle x,v \rangle <0\}$ with its outgoing normal  vector $v$. The set of $\ell_2$ horoballs is then naturally parameterized by the unit sphere $\S^{d-1} \subseteq \R^d$.

For the $\ell_1$ norm on $\Z^d$ the set of horoballs looks really different since it is  countable. It can be checked by considering the horoballs as limits of $\ell_1$ balls with centers going to infinity and  by analyzing the ultimate position of the origin on these $\ell_1$ balls. 
For instance, in $\Z^2$, the horoballs are half-spaces delimited by the diagonal or the anti-diagonal lines, and quarter spaces delimited by integer translations of such half-spaces (see \cref{fig:L1horoball}).
A similar phenomenon occurs for the $\ell_\infty$ norm. However, notice that any horoball for a norm $\ell_p$, $p \in [ 1 , +\infty ]$, is included in a $\ell_2$ horoball.

\begin{figure}
\begin{tikzpicture}
\fill[fill=gray!20] (0,0) -- (2,2) -- (3.5,2) -- (3.5,-2) -- (2,-2) -- cycle;
\draw (0,0) -- (1,1);
\draw (1,1) -- (2,0);
\draw (0,0) -- (1,-1);
\draw (1,-1) -- (2,0);
\draw (0,0) -- (0.5,0.5);
\draw (0.5,0.5) -- (1,0);
\draw (0,0) -- (0.5,-0.5);
\draw (0.5,-0.5) -- (1,0);
\draw (0,0) -- (1.5,1.5);
\draw (1.5,1.5) -- (3,0);
\draw (0,0) -- (1.5,-1.5);
\draw (1.5,-1.5) -- (3,0);
\draw[very thick] (0,0) -- (2,2);
\draw [very thick](0,0) -- (2,-2);

\draw  (0.75,0.75) node {$\times$};
\draw  (0.75,0.75) node[above left] {$O$};
\end{tikzpicture}
\caption{A $\ell_1$ horoball (in gray) for the group $\Z^2$ obtained as limit of balls.}\label{fig:L1horoball}
\end{figure}
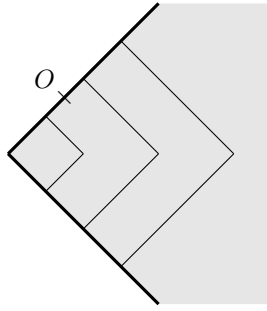

Unlike the euclidean case, the spaces of horoballs are all homeomorphic  when considering hyperbolic distances on a fixed hyperbolic group \cite{Ghys}. 

 In the following, $B_{R}(g)=\{h \in G : d(h,g) <R \}$ denotes the open ball of radius $R>0$ centered at $g \in G$ for the metric $d$ and  its closure, which is compact by assumption,  is denoted by $\overline{B_{R}(g)}$.  It is worth noting that  $B_{R}(g) = B_R(1_G)g $, since the distance is right-invariant. In addition, any  ball  at the origin is symmetric, that is,  $B_{R}(1_G)^{-1} = B_R(1_G)$.
 
Depending on the chosen distance on the group $G$, the horoball could be large, \ie,  contains arbitrary large balls,  or could be even empty. For instance, any horoball of $\Z^d$ or  $\R^d$ for the $\ell_p$ norms, $p\in [1 , + \infty]$, are large. We provide here a criterion that assures to have  large horoballs. This is a strengthening of \cite[Proposition 1.5]{AGW}.  This  applies particularly to horoballs associated with the $\ell_1$ distance of a finitely generated group. 
\begin{prop}
\label{lem:BigHoroball}
    Let $G$ be an infinite group with a proper right-invariant distance. If  a horofunction $j$ satisfies $ \inf_{g \in G} j(g) = - \infty$, then its associated  horoball  contains arbitrarily large balls.  
    
As a consequence, any horoball for the $\ell_1$ distance of a finitely generated group contains arbitrarily large balls.
\end{prop}
\begin{proof}
Let $R>0$ and  $(g_n)_{n\in\N}\subseteq G$ be a sequence such that $b_{g_n}$ converges to $j$. By hypothesis, we can take $g \in G$ such that $j(g) < -4R$. Then,  for any large enough $n \in \N$, $b_{g_n}(g)  \le -2R$. Now, for any $h \in B_{R}(g)$, the triangular inequality gives
$$b_{g_n}(h)= d(h,g_n) - d(g_n, 1_G) \le d(h,g)+ d(g_n,g) - d(g_n, 1_G)\le -R, $$
where the last inequality is true for any large enough $n$. This shows that the ball $B_{R}(g)$ is contained in the horoball $\{ j <0\}$.

To prove the second statement of the lemma, assume that  $G$ is a group generated  by a finite set $\cS$ and  consider its $\ell_1$ distance. Observe that for any sequence $(g_n)_{n\in \N} \subseteq G$ leaving any compact, the element $g_n$ is given by $s_{N_n}^{(n)}s_{N_n-1}^{(n)} \ldots s_{1}^{(n)}$ with  $s_j^{(n)} \in \cS \cup \cS^{-1}$, where $N_n$ is the smallest such integer, \ie,  $N_n = d_{\ell_1}(g_n, 1_G)$. 
For any integer $k>0$, up to taking a subsequence,  the term   $s_k^{(n)}s_{k-1}^{(n)} \ldots s_1^{(n)}$ is constant in $n$. It is then denoted $s_k \ldots s_1$. For any large enough $n$, the very definition of the $\ell_1$ distance provides
$$b_{g_n}(s_k \ldots s_1)  = d_{\ell_1}(g_n , s_k \ldots s_1) - N_n =  d_{\ell_1}(g_n s_1^{-1} \ldots s_k^{-1}, 1_G) -N_n \le -k.$$
Since $k$ is arbitrary,  this shows that any horofunction $j$ that is an accumulation point of $b_{g_n}$ satisfies $\inf_{g\in G} j(g) = -\infty$.
\end{proof}

On the contrary, some distances may define horofunctions with empty horoballs. For instance, consider the free abelian group infinitely generated  by $(e_i)_{i\in \N}$ (i.e., $\bigoplus_{i\in \N}\Z$), equipped with the norm  $\| \sum_{i\in \N} x_i e_i\|= \sum_{i\in \N} |x_i|i $ with $x_i \in \Z$. Then, any horofunction that is an accumulation point of the functions $b_{e_i}$ has a trivial horoball. More dramatically, for the locally finite group  $\bigoplus_{i\in \N}\Z/2\Z$, the same remark provides that any horoball is empty. 

Throughout the proofs, the most interesting case will be when the group $G$ does not have empty horoballs. A compactness argument provides the following finitary characterization of this case.

\begin{lem}\label{lem:NonemptyHoroball}
Let $G$ be a group with a proper right-invariant distance that does not admit empty horoballs. Then, there exists an integer $N>0$ such that 
\begin{align*}
 H \cap B_{N}(1_{G}) \neq \emptyset, \quad \forall H \in \cH.
 \end{align*} 
\end{lem} 
\begin{proof}
For the sake of a contradiction, assume that there is a sequence of horofunctions $ (j_n)_{n\ge 1}$ such that each horoball $\{x: j_n(x) < 0\} $ is disjoint from the ball  $B_{n}(1_{G})$. Let $ j\in  \partial G$  be an accumulation point of  the sequence $(j_n)_{n\geq 1}$.  The  horoball $\{x: j(x) < 0\} $ is then  disjoint  from any ball centered at $1_G$, hence is empty. This contradicts the hypothesis. 
\end{proof}

Our applications in Section \ref{sec:closedness} will mainly focus on locally compact, second countable groups that are subgroups of a Banach space $E$ (hence are abelian). In this case, it is known that any horofunction is bounded from below by a linear functional $L$ of norm at most $1$ \cite[Lemma 3.1]{GouezelKarlsson}. Hence, any horoball is included  in a half-space of the form $L<0$. In this sense, the half-spaces are the largest possible horoballs for $G$. 
For a general group, this geometrical property is still valid but only locally: any horoball is locally tangent to a ball, in the identity element. See  \cref{lem:geomHoroball2} and its illustration  in \cref{fig:geomHoroball2}. 

Also remark that for a countable group $G$, the ball $B_\epsilon (1_G)$ is trivial, \ie, reduced to the neutral element $1_G$, for any  small enough $\epsilon$. So, the reader interested only in countable groups may simplify the formulas in all the following statements. 

\begin{lem}\label{lem:geomHoroball}
Let $G$ be an infinite group  with a proper right-invariant distance $d$. Then, for any  $M>0$  and $\epsilon>0$  there exists an integer $n_{0} \in \N$ such that for  $g \in G$  with $d(g,1_G) \ge n_{0}$ 
one can find a horoball $H\in\cH$ such that
$$ \left[\overline{H} \cap \overline{B_{M}(1_G)}  \right] g \subseteq B_\epsilon (1_G) B_{d(g,1_G)}(1_G).$$   
\end{lem}

\cref{lem:geomHoroball} is a particular case of \cref{lem:geomHoroball2} below, where we do not need to consider all the horoballs but only those that are defined by taking the limit of the points outside a specific subset of $G$. We introduce the following notions and notations.  For an unbounded set  $G_0 \subseteq G$, we denote by $\partial G_0$ the set of horofunctions $\displaystyle j=\lim_{n\to\infty} b_{g_n} \in \partial G$, where $(g_n)_{n\in \N}$ is a sequence in $G_0$ going to infinity.  
Similarly, we denote by $\cH_{\partial G_0}$ the set of horoballs $\{x\in G : j(x)<0\}$ for $j \in  \partial G_0$. Note that \cref{lem:geomHoroball} follows from \cref{lem:geomHoroball2} by taking $A =G$ and $\bar{g} =1_G$.

\begin{lem}\label{lem:geomHoroball2}
Let $G$ be an infinite group with a proper right-invariant distance $d$ and $A \subseteq G$ be an unbounded subset. Then, for any  $M>0$  and $\epsilon>0$  there exists an integer $n_{0} \in \N$ such that for any $\bar{g}\in G$ and $g \in A\bar{g}$  with $d(g,\bar{g}) \ge n_{0}$ one can find a horoball $H\in\cH_{\partial A^{-1}}$ such that
$$ \left[\overline{H} \cap \overline{B_{M}(1_G)}  \right] g \subseteq B_\epsilon (1_G) B_{d(g,\bar{g})}(\bar{g}).$$   
\end{lem}  
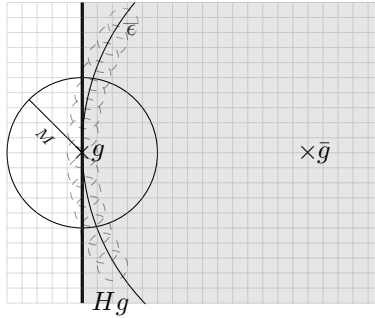
\begin{figure}[h]
\begin{tikzpicture}
\draw [gray!20, fill=gray!20] (0,-2) rectangle (4,2);
\draw (0,-2) node[right] {$Hg$};
\draw (0,0) circle (1) ;
\draw (0,0) -- ++(135:1) node[midway, below, sloped] {\tiny{$M$}};
\draw (0.7,2) arc (140:224:3);
\draw [very thick] (0,-2) -- (0,2);

\draw [very thin, gray, opacity= 0.3] (-1,-2) grid[step=0.2] (4,2);
\draw  (3,0) node {$\times$};
\draw  (3,0) node[right] {$\bar{g}$};
\draw[dashed,opacity= 0.3] (3,0)++(145:3) circle (0.2);  
\draw[opacity= 0.5,very thin] (3,0)++(145:3)-- ++(0:0.2); 
\draw[opacity = 0.8] (0.65, 1.63) node {\small{$\epsilon$}};
\draw[dashed,opacity= 0.4] (3,0)++(150:3) circle (0.2);  
\draw[dashed,opacity= 0.4] (3,0)++(155:3) circle (0.2);
\draw[dashed,opacity= 0.4] (3,0)++(160:3) circle (0.2);
\draw[dashed,opacity= 0.4] (3,0)++(165:3) circle (0.2);
\draw[dashed,opacity= 0.4] (3,0)++(170:3) circle (0.2);
\draw[dashed,opacity= 0.4] (3,0)++(175:3) circle (0.2);
\draw[dashed,opacity= 0.4] (3,0)++(180:3) circle (0.2);
\draw[dashed,opacity= 0.4] (3,0)++(185:3) circle (0.2);
\draw[dashed,opacity= 0.4] (3,0)++(190:3) circle (0.2);
\draw[dashed,opacity= 0.4] (3,0)++(195:3) circle (0.2);
\draw[dashed,opacity= 0.4] (3,0)++(200:3) circle (0.2);
\draw[dashed,opacity= 0.4] (3,0)++(205:3) circle (0.2);
\draw[dashed,opacity= 0.4] (3,0)++(210:3) circle (0.2);
\draw  (0,0) node {$\times$};
\draw  (0,0) node[right] {$g$};
\end{tikzpicture}
\caption{Illustration of \cref{lem:geomHoroball2} for  the $\ell_2$ distance on $G=\Z^2$.}\label{fig:geomHoroball2}
\end{figure}

\begin{proof}
We first show that the result holds for $\bar{g}= 1_G$. By contradiction, suppose that there exist constants $M>0$ and $\epsilon>0$ and a sequence $(g_{n})_{n\in \N}$ in $A$ going to infinity, such that for any horoball $H \in \cH_{\partial A^{-1}}$ we have that
\begin{align}\label{eq:nonEmpty}
\left[ \overline{H} \cap \overline{B_{M}(1_G)} \right ]  \not\subseteq  B_\epsilon (1_G) B_{d(g_n,1_G)}(g_n^{-1}). 
\end{align} 
 Up to taking a subsequence, we can assume that  $(b_{g_{n}^{-1}})_{n\in \N}$ converges to a horofunction $j \in \partial A^{-1}$. By assumption,  the horoball  $H=\{j <0\} \in \cH_{\partial A^{-1}}$ satisfies \eqref{eq:nonEmpty}. 
 By a compactness argument, there exists an element $x\in\overline{H}\cap \overline{B_{M}(1_G)}$ such that for infinitely many $n \in \N$, $d(x, B_{d(g_n,1_G)} (g_n^{-1})) \ge \epsilon$. Take an element $z \in H$ such that $d(x,z)<\epsilon/2$. For any large enough $n \in \N$ we have that $b_{g_n^{-1}}(z)<0$
(\ie, $z\in B_{d(g_n,1_G)} (g_n^{-1})$), but  by the choice of $x$, for infinitely many $n\in \N$ we have that $d(z, B_{d(g_n,1_G)} (g_n^{-1})) \ge \epsilon/2$. This is a contradiction and we therefore obtain the result for $\bar g =1_G$.

To conclude the general case for $\bar{g} \in G$ and $g \in A\bar{g}$, it is enough to use the case $\bar{g}=1_G$ taking $g\bar{g}^{-1}$ in place of $g$. 
\end{proof}

The next geometrical lemma is needed for the directed version of our main result. Roughly speaking, it states that  some translation  of a  ``cone'' $G_0$  has a  larger  intersection  with a  ball around the origin than that of the untranslated cone $G_0$. See \cref{fig:geomHoroball3} for an illustration. 
 
\begin{lem}\label{lem:geomHoroball3}
Let $G$ be an infinite group  with a proper right-invariant distance $d$, $G_0 \subseteq G$ an unbounded subset and $\eta >0$. Then, for any  $g \in G$ such that $j(g^{-1})<-\eta $ for any $j \in \partial G_0$, there exists an integer $n_{1}\in \N$ such that for any real number $r\ge n_1$
$$\left[G_0 \cap B_{r+\eta}(1_G)\right] g  \subseteq  B_{r}(1_G).$$  
\end{lem}

\begin{proof}
The proof is by contradiction. If the result does not hold, there exist $g\in G$ with $j(g^{-1})<-\eta$ for all $j\in \partial G_0$ and sequences $(g_n)_{n\in \N}$ in $G_0$ and $(r_n)_{n\in \N}$ in $\R$ going to infinity such that $d(g_n,1_G)<r_n+\eta$ and $d(g_ng, 1_G) \ge r_n$ for all $n\in \N$. This means that $b_{g_n}(g^{-1}) \ge - \eta$. Since the elements $g_n$ go to infinity, any accumulation point $j \in \partial G_0$ satisfies $j(g^{-1})\ge -\eta$. This is a contradiction.    
\end{proof}

\begin{figure}[h]
\begin{tikzpicture}
\draw (0,0) circle (2) ;
\draw (0,0) -- ++(0:2) node[midway, above] {$r$};

\draw  (0,0) node {$\times$};
\draw  (0,0) node[below left] {$0$};
\draw [very thick] (0,0) -- (2,2);
\draw [very thick] (0,0) -- (2,-2);
\draw (2.5,-2) node {$G_0$};
\fill[gray!20, opacity=0.4] (0,0)--(2.2,2.2)--(2.7,2.2) -- (2.7,-2.2) -- (2.2,-2.2) -- cycle;
\coordinate (g) at (-1,0);
\draw  (g) node {$\times$};
\draw  (g) node[below left] {$g$};
\draw[dashed]  (g) circle (2.3) ;
\draw[dashed] (g) -- ++(160:2.3) node[midway,  above, sloped ] {$r+ \eta$};
\draw [dashed, very thick] (g) -- ++(2,2);
\draw [dashed, very thick] (g) -- ++(2,-2);
\draw (1.5,-2) node {$G_0g$};
\end{tikzpicture}
\caption{Illustration of \cref{lem:geomHoroball3} for $G= \Z^2$. }\label{fig:geomHoroball3}
\end{figure}
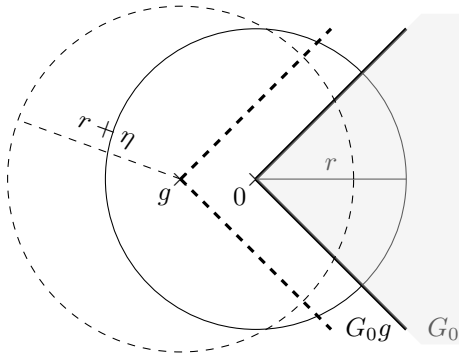

\section{Robinson Crusoe theorem}
\label{sec:Robinson_Crusoe}

In this section we prove our main abstract theorems that will allow us in the next sections to establish some extensions of Schwartzman's theorem and other applications.  
We recall that a {\em topological dynamical system} $(X, T, G)$ is given by a continuous left-action $ T \colon G \times X \to X$ of a group $G$ on a compact metric space $X$ equipped with a distance ${\rm dist}$. This provides a family of  self-homeomorphisms $\{T_{g}: g\in G\}$ of $X$ such that $T_{g}\circ T_{h} = T_{gh}$ for any $g,h \in G$ and the maps $T_g$ depend continuously (for the uniform topology) on the element $g$. A subset $Y\subseteq X$ is {\em $G$-invariant} (or simply {\em invariant} when there is no ambiguity) if $T_g(Y)=Y$ for all $g\in G$. 

The next theorem is \cref{theo1:SchwartzmannVersionAbstraite} from the introduction.

\begin{theo}[Robinson Crusoe theorem]\label{theo:SchwartzmannVersionAbstraite} Let $ (X,T,G)$ be a topological dynamical system and  $G $ be an infinite group with a proper right-invariant distance. Let $O \subsetneq X$ be an open, not closed and $G$-invariant ($T_g(O) = O, \forall g \in G$) subset of $X$. Then, for any neighborhood $U$ of the boundary of $O$ there exists a horofunction $j \in \partial G$ such that its associated  horoball $H$ in $G$ satisfies 
$$O \cap \bigcap_{g \in H}T_g^{-1}(U)\neq \emptyset.$$
\end{theo}
Equivalently, \cref{theo:SchwartzmannVersionAbstraite} states that for any closed invariant subset $F\neq X$, which is not isolated (i.e., it is not open), there exists a point outside of $F$ which remains close to $F$ under the action along a horoball $H$.

Notice that  the distance (or even the semi-distance) on the group $G$ can be chosen independently of the action. 
The theorem is worthwhile for a distance where any horoball  is large, like in the $\ell_1$ distance in a finitely generated group.

The idea of the proof  for  $G =\Z$ is as follows. By contradiction, assume  no backward or forward
orbit stays in a neighborhood of the border of  $O$. A compactness argument provides a  finitary version  as follows: there exists a constant $M>0$ such that no point in $O$ can have $M$ consecutive iterates staying close to the border $\partial O$. Since $O$ is not closed and invariant, any point close enough to the border $\partial O$ will  then contradict the constant $M$. The proof  for the general case is similar but by considering  the compact set of horoballs instead  of  the half-spaces $(0, +\infty)$ or  $(-\infty,0)$. The geometrical fact that the horoballs are locally tangent to a ball at the identity (\cref{lem:geomHoroball}) helps to go from the finitary version to the infinite one.

\begin{proof}
Observe that the claim is obvious if $G$ admits an empty horoball. So we assume that all horoballs are nonempty. 
This allows defining the non-negative constant $d_{0}$ by 
$$ d_{0} = \inf_{x \in O, H  \in \cH} \sup_{g\in H} {\rm dist}(T_{g}(x), \partial O).$$
 Our goal is to show that $d_0=0$. For the sake of contradiction, assume that $d_{0}>0$.  
Since the action is continuous and $B_1(1_G)$ is precompact, the family of maps $\{T_h : h \in B_1(1_G)\}$ is relatively compact and hence there exists $\delta_0>0$ such that  ${\rm dist}(x, y) < \delta_0$ implies that ${\rm dist}(T_h (x), T_h(y)) < d_0/2 $ for all $h \in B_1(1_G)$. 

For any 
$0<\epsilon < \min\{d_0,\delta_{0},1\}$ and $N$ the integer given by \cref{lem:NonemptyHoroball}, set $M_\epsilon$ to be the constant 
\begin{equation*}
\begin{split} 
M_\epsilon = \sup \bigl\{n \ge N :  \exists H \in \cH, \exists x \in X \textrm{ with }{{\rm dist}}(x,\partial O) \ge \epsilon,  \textrm{ s.t. } \forall g\in H \cap  {B_{n}(1_G)}, \\ \exists h \in B_\epsilon(1_G), ~   {\rm dist}(T_{hg}(x), \partial O) < \delta_0 \bigr\}.
\end{split} 
\end{equation*}

By continuity of the action, the set from which we take the supremum is not empty for any $\epsilon>0$ small enough. Moreover, we claim that $M_{\epsilon}$ is finite. Indeed, if this is not the case, then for infinitely many integers $n$ we may find $H_n\in \cH$ and $x_n\in X$ such that $\dist(x_n,\partial O)\geq \epsilon$, and for all $g\in H_n\cap B_n(1_G)$, $\dist(T_{hg}(x_n),\partial O)<\delta_0$ for some $h\in B_{\epsilon}(1_G)$. Since $\partial O$ is invariant by the action, 
$\dist(T_g(x_n),\partial O)<d_0/2$ for all $g\in H_n\cap B_n(1_G)$. We may write $H_n=\{j_n<0\}$ and, taking a subsequence if needed, we can assume $x_n$ and $j_n$ converge to some $x$ and $j$ respectively as $n\to \infty$. Let $H=\{j<0\}$ and take $g\in H$. For $n$ large enough we have that $g\in H_n\cap B_n(1_G)$ and therefore $\dist(T_g(x_n),\partial O)<d_0/2$. This implies that $\dist(T_g(x),\partial O)\le d_0/2$. This inequality is in conflict with the definition of $d_0$. We conclude that $M_{\epsilon}<\infty$. 

The continuity of the action and the compactness of $\partial O$ provide the existence of $0<\delta < \epsilon$ such that  
\[ {\rm dist}(x,\partial O) < \delta ~ \text{ implies  that }  \sup_{g\in\overline{ B_{n_0}(1_G)}}   {\rm dist}(T_{g}(x),\partial O) <\epsilon,\] 
where $n_0 $  is the integer given by \cref{lem:geomHoroball} associated with the constants $M = M_{\epsilon}+1$ and $\epsilon$.

We will show by contradiction that any point  $x$ such that ${\rm dist}(x,\partial O) < \delta$ satisfies ${\rm dist}(T_{g}(x),\partial O) <  \epsilon$ for any $g \in G$.
Otherwise, assume that the set
\begin{equation} \label{eq:set_equation_inf_radius} \{ d(g,1_G) \colon  g \in G \textrm{ s.t. } {\rm dist}(T_{g}(x),\partial O) \ge \epsilon   \}\end{equation} 
is nonempty for some $x$ with ${\rm dist}(x,\partial O) < \delta$. Let $g_* \in G$ be an element realizing the infimum of this closed set. 
By definition of $\delta$, $d(g_*,1_G)> n_0$
 and \cref{lem:geomHoroball} provides  a  horoball  $H \in \cH$ such that 
$$H \cap B_{M_{\epsilon}+1 }(1_G) \subseteq B_\epsilon (1_G) B_{d(g_*,1_G)}(g_*^{-1})= B_\epsilon (1_G) B_{d(g_*,1_G)}(1_G)g_*^{-1}.$$
Thus, for  any $h \in H \cap B_{M_{\epsilon}+1}(1_G)$ ($\neq \emptyset$ by \cref{lem:NonemptyHoroball})  there exists $h_\epsilon  \in  B_\epsilon (1_G) $ such that  $g = h_\epsilon h g_*$ belongs to $B_{d(g_*,1_G)}(1_G)$. 
Since $d(g_*, 1_G)$ is the infimum of the set in \eqref{eq:set_equation_inf_radius}, we have that $ {\rm dist}(T_{g}(x),\partial O) < \epsilon$.
Hence, we get that   
$ {\rm dist}(T_{h_\epsilon h} (T_{g_*}(x)),\partial O) < \epsilon$ for any $h \in  H \cap B_{M_{\epsilon}+1}(1_G)$. The definition of $M_\epsilon$ imposes that  ${\rm dist}(  T_{g_*}(x),\partial O) < \epsilon$, contradicting the definition of $g_*$.

Hence, for any $x \in O$ with ${\rm dist}(x, \partial O) < \delta$ (such an $x$ exists because $O$ is not closed) we get   ${\rm dist}(T_{g}(x), \partial O) < \epsilon < d_{0}$ for any $g \in G$, which is in contradiction with the definition of  $d_{0}$.
 \end{proof}
 
Remark that in \cref{theo:SchwartzmannVersionAbstraite}, the horoball $H$ depends on $O$ (and $U$).  It is natural to wonder if in this theorem one could fix the horoball $H$ beforehand and show that $O\cap \bigcap_{g\in H} T_g^{-1}(U)\neq \emptyset$. This is not the case, even for $G=\Z$.

\begin{example}\label{rem:NSex} Consider the dynamical system given by a $\Z$-action of the time-1 map of a gradient system with only two fixed points on the circle. It is a homeomorphism with two fixed points, one (the south point) is attracting and the other (the north point $N$) is repulsive.  Taking $O$ to be the complement of the north point (and $U$ a small open ball containing it), no accumulation point for forward orbits can be in the neighborhood of the north point. \end{example}

Actually, \cref{rem:NSex} illustrates the main obstruction to finding a horoball that satisfies the conclusion of Robinson Crusoe theorem, within a prescribed set of horoballs: there is a repulsion property. We will show that without the existence of repulsive sets, it is possible to restrict  {\it a priori} the possible directions of the horoball. \cref{thm:directed_RC} will provide  a directed version of Robinson Crusoe theorem.

To be more precise, we introduce the notion of {\em repulsion} of a set for a semigroup. 
For a set $S \subseteq G$, we let $\langle S \rangle_+$ denote the semigroup generated by the elements of $S$, \ie, the set of elements of the form $s_1\cdots s_n$, where each $s_i$ belongs to $S\cup\{1_G\}$.

\begin{defi} Let $(X,T,G)$ be a topological dynamical system, $S \subseteq G$,  and $O \subseteq X$ an open $G$-invariant subset. We say that $O$ is pointwise $S$-repulsed by its border if there exists $\delta>0$ such that for any finite set $F\subseteq O $, 
$$
\dist(T_gF,\partial O) \le \delta \textrm{ for at most finitely many } g \in \langle S \rangle_+.$$
\end{defi}
This means that a finite set of points in $O$ can be separated from $\partial O$ under the iteration of $T_g$ for some $g\in \langle S \rangle_+$. For instance,  for the north-south system in \cref{rem:NSex}, the set  $\S^1\setminus \{N\}$ is pointwise $\{ 1\}$-repulsed by its border, but not pointwise $\{- 1\}$-repulsed.

 In the following, we will be interested in cases of $S$ and $O$, where this property fails, \ie, for any $\delta>0$ there exists a finite set $F \subseteq O$ such that for infinitely many $g \in \langle S \rangle_+$ there exists an element $x_g \in F$ satisfying $\dist(T_g(x_g), \partial O) \leq \delta$.  In such case, we say that $O$ is  {\em not pointwise $S$-repulsed by its border}. Notice that in this case, the pigeonhole principle ensures the existence of a point $x\in F$ such that $x_g= x$ for infinitely many $g\in \langle S \rangle_+$.

Under this hypothesis, for some subset $S$ of $G$ with  good properties, one can state the following Robinson Crusoe theorem restricting the set of possible horoballs.

\begin{theo}[Directed Robinson Crusoe theorem] \label{thm:directed_RC} Let $(X,T,G)$ be a topological dynamical system where $G$ is an infinite group with a proper right-invariant distance. 
Let $G_0 \subseteq G$ be an unbounded subset with an unbounded  complement $G \setminus G_0$.

Assume that 
\begin{itemize}
    \item $O \subsetneq X$ is an open, not closed,  $G$-invariant  subset of $X$.
    \item  The set $O$ is not pointwise $\tilde{S}$-repulsed by its border for some finite subset $\tilde{S} \subseteq \cap_{H \in\cH_{\partial G_0}} H$.   
\end{itemize}
Then, for any neighborhood  $U$ of the border $\partial O$  there exists a  horofunction $j \in \partial [G\setminus  G_0^{-1}]$ such that its associated horoball $H\in \cH_{\partial [G\setminus  G_0^{-1}] }$ satisfies
$$O \cap \bigcap_{g \in H}T_g^{-1}(U)\neq \emptyset.$$
\end{theo}
Notice that the hypothesis of the existence of a set not pointwise $\tilde{S}$-repulsed by its border  for some  finite subset $\tilde{S}$ clearly implies restrictions on the dynamics but also algebraic restrictions on the acting group $G$. 
For instance, if $G$ is a locally finite group (like $\bigoplus \Z/2\Z$), any finitely generated semi-group  $\langle \tilde{S}\rangle_+$  is finite and then no $G$-action can have a set not pointwise $\tilde{S}$-repulsed by its border. 

The proof follows the same strategy as that of  \cref{theo:SchwartzmannVersionAbstraite} but taking into consideration the restriction on the possible set of horoballs. Here  the passage from the finite to the infinite version is ensured by the directed geometrical  \cref{lem:geomHoroball2}. \cref{lem:geomHoroball3} enables us to get a large enough set of iterations that stay close to the border, thus getting a contradiction. 
 
\begin{proof} To simplify the notations, we set $\cH_0 =\cH_{\partial [G\setminus  G_0^{-1}] }$.  Notice that the result is obvious if there is an empty horoball. So we assume that all horoballs in  $\cH_{0 }$ are nonempty. A similar contradiction argument (we leave the details to the interested reader) as in \cref{lem:NonemptyHoroball}, shows that there exists an integer $N>0$ such that  $H\cap B_N (1_G) \neq \emptyset$ for every horoball $H \in \cH_{0}$.
Let $d_{0}$ be the non-negative constant
$$ d_{0} = \inf_{x \in O , H  \in \cH_{0}} \sup_{g\in H} {\rm dist}(T_{g}(x), \partial O).$$
For the sake of contradiction, assume that $d_{0}>0$.  
Since the action is continuous and $B_1(1_G)$ is precompact, the family of maps $\{T_g: g \in B_1(1_G)\}$ is relatively compact and hence there exists $\delta_0>0$ such that  ${\rm dist}(x, y) < \delta_0$ implies that ${\rm dist}(T_g (x), T_g(y)) < d_0/2 $ for all $g \in B_1(1_G)$. 

For any $0<\epsilon < \min(\delta_{0},1)$, set $M_\epsilon$ to be the constant 

\begin{equation*}
\begin{split} 
M_\epsilon= \sup \bigl\{n \ge N :  \exists H \in \cH_{0}, x \in X \textrm{ s.t. }{{\rm dist}}(x,\partial O) \ge \epsilon,  \forall g\in H \cap  {B_{n}(1_G)}, \\ \exists h \in B_\epsilon(1_G), ~   {\rm dist}(T_{hg}(x), \partial O) < \delta_0 \bigr\}.
\end{split} 
\end{equation*}

By continuity of the action, the set from which we take the supremum is not empty for any small enough $\epsilon>0$. Moreover, we claim that $M_{\epsilon}$ is finite. Indeed, if this is not the case, then for infinitely many integers $n$ we may find $H_n\in \cH_{0}$ and $x_n\in X$ such that $\dist(x_n,\partial O)\geq \epsilon$, and for all $g\in H_n\cap B_n(1_G)$, $\dist(T_{hg}(x_n),\partial O)<\delta_0$ for some $h\in B_{\epsilon}(1_G)$. It follows that 
$\dist(T_g(x_n),\partial O)<d_0/2$ for all $g\in H_n\cap B_n(1_G)$.  We may write $H_n=\{j_n<0\}$ and taking a subsequence if needed, we can assume that $x_n$ converges to some $x$ and $j_n$ to some $j$. Then the horoball $H=\{j<0\}$ belongs to $\cH_0$. Take $g\in H$. For $n$ large enough we have that $g\in H_n\cap B_n(1_G)$ and therefore $\dist(T_g(x_n),\partial O)<d_0/2$. This implies that $\dist(T_g(x),\partial O)\le d_0/2$. This inequality is in conflict with the definition of $d_0$. We conclude that $M_{\epsilon}<\infty$. 

Let $n_0$  be the integer given by \cref{lem:geomHoroball2} associated with the constants $M = M_{\epsilon}+1$, $\epsilon$ and the set $A= G \setminus G_0$. 
Let $S = \tilde{S}^{-1}$, that is,  a finite set in $\{g \in G: \forall j \in \partial G_0, \ j(g^{-1}) <0 \}$ such that  $O$ is not pointwise $S^{-1}$-repulsed  by its border. Set $\eta$ to be $-\sup_{s\in S, j \in \partial G_0} j(s^{-1}) >0$. Let $n_1$ be a constant provided by \cref{lem:geomHoroball3} for each element $g$ in the finite set $S$. Furthermore, taking $n_1$ large enough, we may assume   $\overline{B_{n_0}(s)} \subset  B_{n_1}(1_G)$ for any element $s \in S$.

The continuity of the action and the compactness and invariance of $\partial O$ provide the existence of $0<\delta < \epsilon$ such that  
\[ {\rm dist}(x,\partial O) < \delta ~ \text{ implies  that }  \sup_{g\in\overline{ B_{n_1}(1_G)}}   {\rm dist}(T_{g}(x),\partial O) <\epsilon.\] 
\medskip

We will show that there is a set $C \subseteq G$ containing arbitrary large balls such that  
  ${\rm dist}(T_{g}(x),\partial O) <  \epsilon$  for any $g\in C$ and for any point $x \in O$  satisfying ${\rm dist}(x,\partial O) < \delta$. 
To prove it, consider for a point $x$ at distance less than $\delta$ to $\partial O$, the set 
\[
C(x)  =\{g \in G: {\rm dist}(T_{g}(x),\partial O) < \epsilon   \}.
\]
By the choice of $\delta$, this set $C(x)$ contains the ball $B_{n_1}(1_G)$. 
First, we claim it satisfies the following property:
\begin{align}\label{eq:propStable} 
\textrm{If for } g \in G, \exists H \in \cH_0 \textrm{ s.t. } \left[H\cap B_{M_\epsilon +1}(1_G)\right]g \subseteq B_\epsilon(1_G)C(x), \textrm{ then } g \in C(x). \tag{\texttt{P}}
\end{align}
This comes from the fact that if for any $h $ in $H \cap B_{M_\epsilon +1} (1_G)$ (nonempty by the definition of $N$)  there is ${h}_\epsilon \in \left(B_\epsilon(1_G)\right)^{-1}$ such that ${\rm dist}(T_{{h}_\epsilon h}(T_{g} x),\partial O) < \epsilon $.  The definition of $M_\epsilon$ implies then that ${\rm dist}(T_{g}(x),\partial O) < \epsilon $, \ie, $g \in C(x)$. 

Assume that $C(x)$ contains a ball $B_{R} (\bar{g})$ of radius $R \ge n_1$. We claim that for any $s \in S$ the  set $C(x)$ also contains the ball of greater radius $B_{R+\eta}(s\bar{g})$. 
Otherwise, there is $h^* \in B_{R+\eta}(s\bar{g}) \setminus C(x)$ minimizing  the distance $d(h^*, s\bar{g})$. This minimality condition implies $B_{d(h^*,s \bar{g})} (s\bar{g})\subset C(x)$.
Also, notice that by the choice of the constant $n_1$,  the ball $\overline{B_{n_0}(s\bar{g})}$ is included in $B_{R}(\bar{g})$, hence in $C(x)$. So we have  $d(h^*, s\bar{g}) > n_0$.  
If  $h^*$ belongs to $[G\setminus G_0]s\bar{g}$, \cref{lem:geomHoroball2} ensures the existence  of a horoball $H \in \cH_0$ such that 
$$\left[H \cap B_{M_\epsilon+1}(1_G)\right] h^*  \subseteq B_\epsilon (1_G) B_{d(h^*,s \bar{g})}(s\bar{g}) \subseteq  B_\epsilon (1_G) C(x).$$ 
This is impossible by Property \eqref{eq:propStable}.
It follows that $h^*$ belongs to $ G_0s\bar{g}$. Then, \cref{lem:geomHoroball3} applied to $g=s$,  implies that $h^*$ belongs to $B_{R}(\bar{g}) \subseteq C(x)$: again a contradiction.  

Applying inductively the former claim, we get for any sequence $(s_n)_{n\geq 0}$ in $S$ that $C(x)$ contains the set $\bigcup_{n\geq 0} B_{n_1+n\eta}(s_n \cdots s_1)$, which contains arbitrarily large balls. Therefore $C(x)$ contains the set $C=\bigcup_{n\geq 0}\bigcup_{s_1,\ldots,s_n\in S} B_{n_1+n\eta}(s_n \cdots s_1)$.  Hence any $x\in O$ such that $\dist(x,\partial O)<\delta$, satisfies 
\[ \dist(T_g(x),\partial O) <\epsilon  \  \text{ for all } g \in C.  \]
 Given that the set $O$ is not pointwise $S^{-1}$-repulsed by its border, there exists $x \in O$ such that $\dist(T_{\tilde{g}} (x), \partial O) < \delta$  for infinitely many  $\tilde{g} \in \langle S^{-1} \rangle_+$. For any such $\tilde{g}$, it follows that $\dist(T_g (x), \partial O) < \epsilon$, $\forall g \in   C \tilde{g}$.
Since any $\tilde{g}$ is of the form $(s_n \cdots s_1)^{-1}$, with $s_i\in S$, we get that $\dist(T_h (x), \partial O) < \epsilon$ for any $h$ in $B_{n_1+n \eta} (1_G)$. Since this is true for infinitely many $\tilde{g}$, this remains true  for infinitely many integers $n$, hence for any $h$ in $G$. This contradicts the definition of $d_0$.
\end{proof} 

\section{Extensions of Schwartzman theorem}\label{sec:asymptotics}

In this section we apply the Robinson Crusoe theorems to get several extensions of a classical result due to  Schwartzman in his Ph.D. Thesis \cite{Sch} to general group actions. Recall that Schwartzman's theorem states that given a homeomorphism $T\colon X\to X$ of an infinite compact metric space $X$ and $\epsilon >0$ there exist positive (resp. negative) $\epsilon$-asymptotic pairs. That is, different points $x,y \in X$ such that ${\rm dist}(T^nx,T^ny)\leq \epsilon$ for all $n\geq 0$ (resp. $n\leq 0$). 
This is equivalent to say that there is no positively (or negatively) expansive homeomorphisms on an infinite compact space.  
This result can be seen as a version in topological dynamics of the classical Morse-Hedlund's theorem on the complexity of an infinite subshift \cite{MH}. 

A crucial issue in extending Schwartzman's theorem to general group actions is to define a good notion of asymptoticity that extends the one of asymptotic pair for $\Z$-symbolic systems. We propose a notion using the elements of geometric group theory that we developed in the last sections.
We recall that a factor map between topological dynamical systems $(X,T,G)$ and $(Y,S,G)$ is a continuous and onto map $\pi\colon X\to Y$ such that $S_g \circ \pi = \pi \circ T_g$ for all $g\in G$.
\begin{defi}
For  two dynamical systems $(X,T, G)$ and $(Y,S,G)$, a factor map $\pi \colon X \to Y$, $\epsilon>0$ and a horoball $H$, we call an $(\epsilon, H)$-asymptotic pair relative to $\pi$ any pair $(x,y)$ of two different points $x,y \in X$ such that $\pi(x) = \pi(y)$ and  ${\rm dist} (T_g x, T_gy) \leq \epsilon$ for all $g \in H$. 

\noindent When the system $Y$ is trivial we simply say that $(x,y)$ is an $(\epsilon, H)$-asymptotic pair. 
\end{defi}

We might naively expect from Schwartzman's theorem that any group action admits $(\epsilon,H)$-asymptotic pairs for each horoball $H$. But this is false even for $\Z^2$-actions, as shown by Ledrappier's shift example. 
\begin{example}\label{ex:Ledrappier}(Ledrappier's example)
Consider the $\Z^2$ shift
\[X = \{ (x_{i,j})_{(i,j) \in \Z^2} \in (\Z/2\Z)^{\Z^2}:x_{i,j}+x_{i+1,j}+x_{i,j+1} = 0 \mod 2 \quad \forall i,j \in \Z\},\]
and the $\Z^2$-action generated by the horizontal and vertical shifts.
If $v \in \R^2$  is not an  outgoing  vector normal to  a face of the unit simplex in $\R^2$, then for each $x \in X$, the coordinates of $x$ within $H_v =\{z : \langle z,v \rangle < 0\}$ determine all the coordinates of $x$. So, there is no $(\epsilon, H_v)$-asymptotic pair for any small enough $\epsilon>0$. Conversely, for each of the three  normal outgoing vectors $v$ of the face of the unit simplex,  there is an asymptotic pair for the half-space $H_v$. See \cite[Example 2.7]{BD} for details. 
\end{example}

Nevertheless, an important step in extending Schwartzman's theorem to $\Z^{d}$-actions has been done by Boyle and Lind through the notion of nonexpansive direction \cite{BD}. In contrast with the $\Z$-case, their result provides the existence of $(\epsilon,H)$-asymptotic pairs only for some horoball $H$. Technically, it ensures the existence of some nonexpansive directions from which the former claim can be easily deduced.  

\subsection{Existence of $(\epsilon,H)$-asymptotic pairs}\label{sec:ExistenceAsymptotic}
Thanks to  \cref{theo:SchwartzmannVersionAbstraite} we propose a first extension of Schwartzman's theorem (more in the spirit of Boyle and Lind's theorem) to any countable group and more generally to any second countable and locally compact group. Moreover, our extension is relative to a factor map.

 A factor map $\pi$ is {\em bounded to one} (resp. {\em constant to one}) if the cardinality of the $\pi$-fibers is uniformly bounded (resp. equals to a constant).  
 
\begin{theo}\label{theo:factors_general1}
Let $(X,T,G)$ and  $(Y,S,G)$ be topological dynamical systems, where $G$ is an  infinite group with a proper right-invariant distance, and let $\pi \colon X \to Y$ be a factor map. Then, either  
\begin{itemize} 
\item $\pi$ is bounded to one, 
\item or for any $\epsilon>0$
 there exist a horoball $H \in \cH$ and an $(\epsilon,H)$-asymptotic pair $(x,y) \in X^2$ relative to $\pi$. 
\end{itemize} 
In particular, when  $(Y,S,G)$ is trivial and $X$ is infinite, there exists an $(\epsilon,H)$-asymptotic pair $(x,y) \in X^2$. 
\end{theo}
Note that in the case of a trivial $(Y,S,G)$, \cref{theo:factors_general1} recovers \cref{thm_1_intro} announced in the introduction.

Interestingly, in  the context of symbolic dynamics, \cref{theo:factors_general1} provides the existence of  an asymptotic pair for any infinite subshift.  More precisely, for an infinite closed subset $X$ of $A^G$ (where $A$ is a finite set and $G$ is a countable group), invariant by the natural shift action $T$ induced by translations on $G$, there exists a horoball $H$ and two points $x \neq y \in X$ that coincide when restricted to $H$.    

Remark that \cref{theo:factors_general1} does not say that its items are mutually exclusive. In fact, many situations can arise: there exist factor maps that are bounded to one but do not identify asymptotic pairs ({e.g.}, consider an irrational rotation on the complex circle with the factor map $z \mapsto z^n$). 
 On the other hand, an example of an  unbounded to one factor map identifying asymptotic points is also given by a dynamic on the circle. For instance, consider a Denjoy's homeomorphism. It preserves a Cantor set $C$. It is standard that it has positively asymptotic pairs: namely the extreme points of a connected component of the complementary of $C$. Moreover, the homeomorphism factorizes onto an irrational rotation by a factor map identifying  these connected components. Finally, recall that there exist Denjoy's examples where the restrictions on the Cantor set $C$ provide  factor maps that are both bounded to one and identify the asymptotic pairs. 
 
Let us stress that  \cref{theo:factors_general1} is intrinsic to the system and does not depend on the embedding into another space with some constraints (like the dimension, connectivity, ...). This differs, for instance, from what the usual study of  minimal sets  of homeomorphisms (like the one of Denjoy homeomorphisms) would suggest.

\begin{proof}
Consider the system $(R_{\pi}, T^{(2)},G)$, where $R_\pi=\{(x,y)\in X\times X: \pi(x)=\pi(y)\}$ and $T^{(2)}$ is the diagonal action $T^{(2)}_g(x,y)= (T_g(x),T_g(y))$. Let $O=R_\pi\setminus\Delta_X$, where $\Delta_X$ is the diagonal on $X$. Clearly, $O$ is open and $G$-invariant. If $O$ is closed, then there exists $\delta>0$ such that ${\rm dist}(x,y) \geq \delta$ for any $(x,y) \in O$. Then the compactness of $X$  implies that $\pi$ is bounded to one. 

Now assume $O$ is not closed. This implies that $\emptyset \neq \partial O \subseteq \Delta_X$. For any $\epsilon >0$, set $U= \{ (x,y) \colon {\rm dist}(x,y) < \epsilon \}$ an open neighborhood of $\partial O$. 
By \cref{theo:SchwartzmannVersionAbstraite}, for any $\epsilon>0$ there exist a horoball $H$ in $\cH$ and different points $x,y \in X$ with $\pi(x)=\pi(y)$ such that for all $g \in H$ we have ${\rm dist} (T_{g}(x),T_{g}(y))< \epsilon$. 
\end{proof}

 Similarly, the application of the Directed Robinson Crusoe theorem provides a directed version of \cref{theo:factors_general1}. This is a second extension of Schwartzman's theorem. For expository reasons, we state a version which is valid for groups admitting non periodic elements and 
where the inverse map is an isometry, i.e.,  $d(x^{-1}, y^{-1}) = d(x,y)$. 
This is the case, for instance, for the $\ell_{1}$ distance of a finitely generated group. The inverse map is also an isometry for any distance $d$ that is {\em bi-invariant}, i.e.,  that is both left and right-invariant: $d(gxh,gyh) = d(x,y)$ for any $g,h,x,y \in G$. This occurs for instance for abelian groups.  

\begin{theo}\label{theo:factors_general}
	Let $\pi \colon X \to Y$ be a factor map between the topological dynamical systems $(X,T,G)$ and  $(Y,S,G)$, where $G$ is an infinite  group with a proper right-invariant distance and where the inverse map is an isometry. Assume that $\pi$ is not bounded to one. 
	Let $k \in G$ be an element generating an unbounded subgroup $\langle k \rangle$.  Then, for any $\epsilon >0$, there exists 
	a horoball $H \in \cH$ such that
	\begin{itemize} 
		\item the horoball  $H$ does  not contain $k$;
        \item there exists an  $(\epsilon,H)$-asymptotic pair $(x,y)$ relative to $\pi$.
	\end{itemize}
	\end{theo}

Observe that the first item implies a strong restriction on the horoballs. For instance, if $G= \Z^2$, the set of $\ell_2$ horoballs that do not contain a given point $k\in \Z^2$ forms a half circle of directions when identified with their normal outgoing vectors (see \cref{sec:closedness} for further discussion). 

Also, notice that following the proof (given below), we can relax the hypothesis on the isometric action of the inverse map
by only considering right-invariant distances. But the price to pay is to get the following more intricate statement.

\begin{theo}[\cref{theo:factors_general} general statement]\label{theo:factors_general2}
Let $\pi \colon X \to Y$ be a factor map between the topological dynamical systems $(X,T,G)$ and $(Y,S,G)$, where $G$ is an infinite group with a proper right-invariant distance. Assume that $\pi$ is not bounded to one. Let $k \in G$ be an element generating an unbounded subgroup $\langle k \rangle$. 
Then, for any $\epsilon >0$, there exists an unbounded  sequence $(g_n)_{n\in \N}$ in $G$ such that 
	\begin{enumerate}
        \item the horoball $\displaystyle H'=\{\lim_{n\to \infty} b_{g_n^{-1}}<0 \}$ does not contain $k$;  
	    \item the horoball $\displaystyle H=\{ \lim_{n\to \infty} b_{g_n}<0 \}$  admits an  $(\epsilon,H)$-asymptotic pair $(x,y)$ relative to $\pi$. 
	\end{enumerate}
\end{theo}
Actually, we believe the correct general setting on the distance to get a similar conclusion is that the inverse map  $g \mapsto g^{-1}$ defines a coarse map (see \cite{CornulierHarpe}). But the statement  would be unnecessarily complicated for our purpose. 

\begin{proof}[Proof of \cref{theo:factors_general}]
The proof starts similarly to that of \cref{theo:factors_general1}. We keep the same notations: \ie, 
we  consider the system $(R_{\pi}, T^{(2)},G)$, where $R_\pi=\{(x,y)\in X\times X: \pi(x)=\pi(y)\}$ and $T^{(2)}$ is the diagonal action $T^{(2)}_g(x,y)= (T_g(x),T_g(y))$. The set  $O$ denotes $R_\pi\setminus\Delta_X$, where $\Delta_X$ is the diagonal in $X\times X$.
The set $O$ is open and  $G$-invariant and the same proof as in \cref{theo:factors_general1} shows it is not closed. 

Moreover, for any $\delta > 0$ there exists $\epsilon >0$ such that if $(x,y)\in O$ and ${\rm dist}(x,y)\leq \epsilon$ then ${\rm dist}((x,y),\partial O)\leq \delta$. Indeed, for the sake of a contradiction, assume this is not true. Taking a subsequence if needed, we may consider a sequence $((x_n,y_n))_{n\in \N}$ in $O$ with $(x_n,y_n)\to (z,z)$ when $n\to \infty$ for some $z\in X$ and ${\rm dist}((x_n,y_n),\partial O)> \delta$ for all $n\in \N$. But, by definition, $(z,z) \in \partial O$, which contradicts the fact that ${\rm dist}((x_n,y_n),\partial O)> \delta$.

Since the group $\langle k \rangle$ is unbounded, both semigroups  $\langle k \rangle_+$ and $\langle k^{-1} \rangle_+$ are unbounded. 
We claim that  $O$ is not pointwise $\{k^{-1}\}$-repulsed by its border.
Indeed, for $\delta>0$ let $\epsilon  >0$ be such that $(x,y)\in O$ and $d(x,y)\leq \epsilon$ implies $d((x,y),\partial O)\leq \delta$.  Let $B_1,\ldots, B_m$ be distinct open balls of diameter $\epsilon/2$ whose union covers $X$. 
Let $x_1,\ldots,x_{m+1}$ be distinct points with $\pi(x_1)=\pi(x_2)=\cdots=\pi(x_{m+1})$. Let $F=\{ (x_i,x_j): i\neq j \} \subseteq O$. Then, by the pigeonhole principle, for any $g\in \langle k^{-1} \rangle_+ \subseteq G$ there exist $i$ and $j$ such that $T_{g}(x_i)$ and $T_{g}(x_j)$ belong to a same ball $B_k$. This means that ${\rm dist}(T_g(x_i),T_g(x_j))\leq \epsilon$, which implies that $(T_g(x_i),T_g(x_j))$ is $\delta$-close to $\partial O$. The claim is proved. 
	
So the set $O$ is non pointwise $\{k^{-1}\}$-repulsed by its border. Set $G_k=\{ g\in G:  k^{-1}\in B_{d(1,g)}(g) \}$. It can be checked that $G_k$ and $G\setminus G_k$ are unbounded since they contain unbounded subsets of $\langle k^{-1} \rangle_+$ and $\langle k \rangle_{+}$ respectively. Moreover, by definition $k^{-1}$ belongs to $\bigcap_{H \in \cH_{\partial G_k}} H $. By the directed Robinson Crusoe theorem \ref{thm:directed_RC} applied to $O$ and $G_0=G_k$, for any $\epsilon>0$ there exist a horoball $H$ in $\cH_{\partial [G\setminus G_k^{-1}]} $ and different points $x,y \in X$ with $\pi(x)=\pi(y)$ such that for all $g \in H$ we have ${\rm dist} (T_{g}(x),T_{g}(y))< \epsilon$.
	
Write  $\displaystyle H=\{\lim_{n\to \infty} b_{g_n}<0\}$ for a sequence $(g_n)_{n\in \N}$ in $G\setminus G_k^{-1}$ going to infinity. Since $g_n^{-1} \notin G_k$ we have that  for all $n\in \N$, $k^{-1}\not\in B_{d(1,g_n)}(g_n^{-1})$.
Since the inverse map is an isometry, this means   that  $ k \not\in  B_{d(1,g_n)}(g_n)$. 
The proof is complete.
\end{proof}

The Robinson Crusoe theorems are valid for second countable, locally compact groups (like  $\R^d$ or Lie groups). 
But in some cases, they hold for trivial reasons. For instance, in the case of a $\R^d$-flow, the notion of $(\epsilon, H)$-asymptotic pair is trivial:  any non fixed  point with any small translation of it by the flow give an asymptotic pair. This problem already occurs in the definition of expansive flows (see \cite{BW}).
This leads to the following question:  
\begin{ques} Provide nontrivial versions of  \cref{theo:factors_general1}  and \cref{theo:factors_general} for flows of (connected)  second countable, locally compact groups actions. 
\end{ques}

\subsection{Structure of the set of asymptotic pairs}

The structure of the set of pairs that are $(\epsilon, H)$-asymptotic  for some horoball $H$ and $\epsilon >0$ depends very much on the particular dynamical system we are working with. It can be really small, for instance finite, and even restricted to only two orbits ({e.g.}, Sturmian subshifts, Denjoy examples of homeomorphism of the circle). Hence it does not form,  in general, a generic set in the topological or measurable frameworks. At the opposite, it can also be really large ({e.g.}, in the fullshift there is a Cantor set of such pairs).  

However, in some cases, the asymptotic pairs reflect features of the global structure of the ambient space.
For instance, for a product dynamical system, it is simple to check that  the set of asymptotic pairs also has a product structure. Less directly, we have the following result for spaces with nontrivial connected components. 

\begin{cor}\label{cor:lemIII} Let $(X,T,G)$ be a topological dynamical system with $\dim X >0$ and $G$ be an infinite   group with a proper right-invariant distance. Then, for any $\epsilon >0$ there exist a horoball $H$ in $G$ and  a connected set $F\subseteq X$ such that 
$0< \textrm{diam}(T_{g}(F)) < \epsilon$ for every $g \in H\cup\{1_{G}\}$. 
\end{cor}
In particular, this result shows, in the case of a dynamic on a connected set, that it admits a (stable) set 
$$W^s_{\epsilon, H}(x) = \{y \in X: \forall h \in H, \ \dist(T_h (x),T_h(y)) < \epsilon   \}$$ of positive topological dimension for some $x \in X$ and some horoball $H \in \cH$.  In differentiable dynamics, this leads to the meaningful notion of stable manifold.
Let us also point out that this result is the starting point in Ma\~n\'e's proof that the minimal components of expansive $\Z$-dynamics are zero-dimensional \cite{M}.  

\begin{proof}Let  $\cX$ denote the set of all closed connected  subsets of $X$. Since $X$ has positive dimension, basics on dimension theory imply that $\cX$ is nonempty. It is also classical that $\cX$ is compact for the  Hausdorff metric. 
Let $O\subseteq \cX$ denote the set of positive diameter connected subsets of $X$. It is clearly invariant, open and non closed. Then, \cref{theo:SchwartzmannVersionAbstraite} gives us immediately the desired conclusion.  
\end{proof}
 
\section{Set of nonexpansive horoballs for a $\Z^d$-action}\label{sec:closedness}

In \cref{sec:asymptotics} we have shown the existence of horoballs admitting $\epsilon$-asymptotic pairs. Here we try to obtain more information about the set of such horoballs. 
For this we follow  the same path initiated by Boyle and Lind in \cite{BD}. In their work the authors consider the notion of  {\em nonexpansive direction} which corresponds in our context to the border of horoballs admitting  $\epsilon$-asymptotic pairs for  the $\ell_2$ distance in $\Z^d$. In this case, horoballs are open half-spaces and their border are  hyperplanes, i.e., directions for $d=2$. Remark that by only considering the border of a half-space, we lose information about its orientation, as the half-space and its opposite share the same border. This is an important difference between our approach and that of nonexpansive directions. 

Boyle and Lind showed in \cite{BD} that the set of nonexpansive directions (under their definition) is a closed subset within the set of all the directions. Furthermore, their work, together with Hochmans's \cite{Hochman:2011} shows that any nonempty closed set can be realized as the set of nonexpansive direction of some $\Z^2$-action.

In this section we restrict ourselves to $\Z^d$-actions and we propose a notion of nonexpansive horoballs (see \cref{def:NDhoroball}, similar to the one of \cite{Z}).  The result  of \cite{BD} about closeness of the set of nonexpansive directions can be easily adapted to the context of $\ell_2$ horoballs. We use this fact to obtain a restriction on the set of nonexpansive  horoballs of a $\Z^d$-action: the intersection of all its nonexpansive horoballs has to be empty (\cref{cor:half_space_Z2}).  This  is an extension of Schwartzman's result.

\begin{defi}\label{def:NDhoroball} For  dynamical systems $(X,T,G)$ and  $(Y,S,G)$, a factor map $\pi \colon X \to Y$ and $\epsilon>0$, a  horoball $H$ is said to be $(\epsilon, \pi)$-nonexpansive if  there exist two different points $x,y \in X$ such that $\pi(x)= \pi(y)$ and ${\rm dist} (T_g x, T_gy) \leq \epsilon$ for all $g \in H$ (\ie, $(x,y)$ is an $(\epsilon,H)$-asymptotic pair relative to  $\pi$). 

When $Y$ is  the trivial system, we simply call the horoball $\epsilon$-nonexpansive.  
\end{defi}
At the opposite, a horoball $H$ that is not $(\epsilon, \pi)$-nonexpansive, satisfies the following property: for every $x,y$ such that $\pi(x) = \pi(y)$,
$$  \sup_{g \in H} {\rm dist}(T_g(x), T_g(y)) \le \epsilon   \text{ implies } x=y.
$$
In this sense, the horoball $H$ determines the points of $X$. For this reason, C. Zinoviadis in \cite{Z} uses the term {\em nondeterministic}  instead of nonexpansive. 

For  $\Z^d$ or $\R^d$-actions, we set the collection of   $(\epsilon,\pi)$-nonexpansive $\ell_2$ horoballs (identified with their normal outgoing vectors) as
\begin{align}\label{eq:NDdef}
\ND (X, \epsilon,  \pi) &=\{ v \in \S^{d-1}: \{x \in \R^d: \langle x,v\rangle < 0\} \textrm{ is } (\epsilon, \pi)\textrm{-nonexpansive}  \}, \\
   \ND (X, \pi) &= \bigcap_{\epsilon>0} \ND (X,\epsilon, \pi). 
\end{align}
Also, for the  trivial factor we simply denote these sets by  $\ND(X, \epsilon)$ and $\ND(X)$.
To avoid confusion between the notions, we refer to the elements of $\ND(X,\pi)$ (resp. $\ND(X)$) as {\em $\pi$-nondeterministic} (resp. {\em nondeterministic}) vectors.  
Similar arguments as the ones in \cite[Lemma 3.4]{BD} show the closeness of the set of nondeterministic vectors for the usual topology on the sphere $\S^{d-1}$.

\begin{lem}\label{lem:BLclosdness}
Let $(X,T, \Z^d)$ be a dynamical system and  $\pi \colon X\to Y$ a factor map onto the system $(Y,S, \Z^d)$. Then, for any $\epsilon >0$ the sets $\ND(X,\epsilon,\pi)$ and  $\ND(X,\pi)$ are  closed subsets of the unit sphere.
\end{lem}
For the reader's convenience, we provide a  proof that $\ND(X,\epsilon,\pi)$ is closed.
\begin{proof} 
First, we assume that for any pair of different points $x, y$ such that $\pi(x) = \pi(y)$, there exists a $g \in \Z^d$ such that $\dist (T_g(x), T_g(y))> \epsilon$. 
Otherwise we have $\ND(X,\epsilon,\pi)= \S^{d-1}$  and there is nothing to prove.

For $v \in \ND(X, \epsilon, \pi)$, there are two points $x,y$ such that $\pi(x) = \pi(y)$, 
 and $\dist(T_g(x), T_g(y))< \epsilon$, for any $g$ in the half-space $H_v = \{g: \langle v, g \rangle<0\}$. By the initial assumption, we may pick $\tilde{g}$, minimizing the distance to $H_v$  such that $\dist (T_{\tilde{g}}(x), T_{\tilde{g}}(y))> \epsilon$. Hence  $\dist (T_{g\tilde{g}}(x), T_{g\tilde{g}}(y))< \epsilon$ for any $g \in H_v$.

We deduce the next characterization:
$v \in \ND(X, \epsilon, \pi)$ if and only if there exist  $x,y \in X$ such that, $\pi(x) = \pi(y),  \dist(x,y) >\epsilon$, 
 and $\dist(T_g(x), T_g(y))< \epsilon$, for any $g$ such that $\langle v, g \rangle<0$. 
 Note that  the converse part is trivial.
From this characterization, it is straightforward to check that $\ND(X, \epsilon, \pi)$ is closed. 

The closeness of  $\ND(X, \pi)$ follows from that of $\ND(X,\epsilon, \pi)$  since it is  the intersection of a nested sequence of such sets.
 \end{proof}

More interestingly, \cref{theo:factors_general1} provides conditions that ensure the nonemptiness of the set $\ND(X, \epsilon, \pi)$ for every $\epsilon>0$. The former argument then implies that the set $\ND(X, \pi)$ is nonempty. This shows in particular that the horoball given by \cref{theo:factors_general1} may be taken uniformly in $\epsilon>0$ for the $\ell_2$ distance. This leads to the following notion of { nonexpansivity along a subset} (similar to \cite{BD}), which becomes relevant when considering horoballs.

\begin{defi}
A topological dynamical system $(X,T,G)$ with $G$ countable is said to be expansive along a subset $H$ if there exists $\epsilon_0>0$ such that 
$$\sup_{g \in H}  {\rm dist} (T_gx,T_gy)\le  \epsilon_0 \text{ implies } x= y.$$ If $H$ fails to meet this condition, the system is said to be non $H$-expansive.  
\end{defi}
Observe that this notion is meaningless in the context of $\R^d$ flows  since any  pair of points $(x,T_\epsilon x)$ for small enough $\epsilon>0$ shows that such flow is non $H$-expansive for any horoball. 

With this definition \cref{theo:factors_general1} can be stated as follows (in the spirit of Schwartzman's result). 
\begin{cor}
Let $(X,T,\Z^d)$ be an infinite topological dynamical system. Then, it is non $H$-expansive for some $\ell_2$ horoball $H$.   
\end{cor}

In addition, the closeness property together with \cref{theo:factors_general} (the directed version on the existence of asymptotic pairs) enables us to guarantee the existence of a nonexpansive horoball (for $\Z^d$ systems) within any given half-space. 
\begin{theo} \label{theo:half_space_nonexpansive} 
Let $\pi \colon X \to Y$ be a factor map between the topological dynamical systems $(X,T,\Z^d)$ and  $(Y,S,\Z^d)$. Assume that $\pi$ is not bounded to one. Then,  any closed  half-space $H\subseteq \R^d$ contains a $\pi$-nondeterministic vector, \ie, 
$$\ND(X,\pi)\cap H \neq \emptyset.$$
\end{theo}
When $X$ is infinite and $Y$ is the trivial factor, we obtain that any closed half-space contains a nondeterministic vector.
\begin{proof}
We endow $\Z^d$ with the $\ell^2$ distance. Let $H\subseteq \R^d$ be a closed half-space and write $H=\{ x \in \R^d : \langle x, c\rangle \geq 0\}$ for some $c\in \R^d$. Let $(g_n)_{n\in \N}$ be a sequence going to infinity such that the distance of $g_n$ to the ray  $\{tc: t\in \mathbb{R}_+ \}$ goes to 0. Let $\cH_n$ denote the set of horoballs that do not contain $g_n$. Identifying a horoball with its outward unit normal vector, we may identify $\cH_n$ with the set $S_n=\{v \in \S^{d-1} : \langle g_n ,v \rangle \geq 0 \}$. Let $\epsilon>0$. 
By \cref{theo:factors_general}, for each $n\in \N$ we may find a vector $v_n \in \ND(X,\epsilon,\pi)\cap S_n$. Taking a subsequence if needed, we may assume that $v_n \to v$ as $n\to \infty$. Since $\ND(X,\epsilon,\pi)$ is closed by \cref{lem:BLclosdness}, we get $v\in \ND(X,\epsilon,\pi)$. We may write $g_n=t_n c + \epsilon_n$ with $\|\epsilon_n\|\to 0$ and $t_n\nearrow \infty$ as $n\to \infty$.  
From $\langle v_n,g_n\rangle \geq 0$ we obtain that $\langle v, c\rangle \geq 0$, and hence $v$ belongs to $H$. As $\epsilon>0$ is arbitrary, using again  \cref{lem:BLclosdness}, we obtain the desired conclusion.

\end{proof}
 
A special case of \cref{theo:half_space_nonexpansive}, for $d=2$ and $Y$ the trivial system, was announced by P. Guillon, J. Kari and C. Zinoviadis in 2015 \cite{GKZ}, but to the best of our knowledge, the proof was not communicated to the community. An immediate corollary of \cref{theo:half_space_nonexpansive} is the following restriction on the set of nonexpansive half-spaces. 
   
\begin{cor} \label{cor:half_space_Z2}
Let $\pi \colon X \to Y$ be a factor map between the topological dynamical systems $(X,T,\Z^d)$ and  $(Y,S,\Z^d)$. Assume that $\pi$ is not bounded to one. Then the origin belongs to the convex hull generated by the elements in $\ND(X,\pi)$.  In particular, the intersection of all the $\pi$-nonexpansive $\ell_2$ horoballs is empty:
$$ \bigcap_{v \in \ND(X)} \{y \in \R^d: \langle y,v \rangle <0   \} = \emptyset.$$
\end{cor}

When  $Y$ is the trivial system we obtain \cref{Thm:intro2} announced in the introduction. 
Recall that $\ND(X)$ is never empty when $X$ is infinite (\cref{theo:factors_general1}). It follows from the empty intersection condition that $\ND(X)$ always contains at least $2$ elements.
More precisely, for $d=2$, \cref{cor:half_space_Z2} implies the following dichotomy. Either  $\ND(X)$ contains at least $2$ independent  vectors $v_1, v_2$ and a third one in the convex cone generated  by $-v_1$ and $-v_2$, or $\ND(X)$  contains only two elements $v_1, v_2$ that are opposites (i.e., $v_1 =-v_2$).
These cases are optimal. For instance, for the Ledrappier's subshift the set $\ND(X)$ consists of the normal outgoing vectors of the faces of a triangle (see \cref{ex:Ledrappier}).
Concerning the second case, Hochman showed in \cite{Hochman:2011} that any direction in $\R^2$ appears as the unique nonexpansive direction of some $\Z^2$-dynamical system. Since any nonexpansive half-space defines a nondeterministic direction and by \cref{cor:half_space_Z2}, this system has only two symmetric nonexpansive half-spaces. 

\begin{proof}
By contradiction, assume that the origin does not belong to the convex hull generated by the set $\ND(X,\pi)$. Note that this set is closed as $\ND(X,\pi)$ is compact (\cref{lem:BLclosdness}). The Hahn-Banach theorem ensures the existence of a half-space $H$ which is disjoint from $\ND(X,\pi)$. \cref{theo:half_space_nonexpansive} then provides a contradiction.

Since the origin belongs to the convex hull generated by the set  $\ND(X,\pi)$, Gordan's theorem provides the interpretation in terms of intersections.
\end{proof}

A natural question is to ask how optimal is  the restriction on $\ND(X)$ in \cref{cor:half_space_Z2}. More precisely, we have the following question:
\begin{ques}\label{ques:realization}
    Given a closed set $N \subseteq \S^{d-1}$ such that the origin  of $\R^d$ belongs to its convex hull, does there exist a topological dynamical system $(X,T, \Z^d)$ such that $\ND(X) = N$?
\end{ques}
Of course, the former question also makes sense when restricted to some class of dynamical systems with specific properties (like transitivity, minimality, subshift of finite type like in \cite{Z}, etc.).

It is already  a challenging question for simple cases: for instance, we do not know if there is a $\Z^2$-action where $\ND(X)$ consists of three vectors, two of them opposites and a third one orthogonal to the others.

Previous results in the literature are not enough to answer the question but provide a partial answer. Recall that Boyle and Lind \cite{BD} and Hochman \cite{Hochman:2011} proved that any closed set of directions can be the set of nonexpansive directions, which we denote by $\NE(X)$, of a topological dynamical system $(X,T,\Z^2)$. For instance, for any vector $u$ determining an (orthogonal)  direction in $\NE(X)$, either $u$ or  $-u$ determines a nonexpansive $\ell_2$ horoball of this action. In particular, the set $\NE(X)$ can be reduced to a point (even defining a line with irrational slope).
Actually, let  $\pi\colon \S^{1} \to \R {\textbf{\rm P}}^1$ denote the covering onto the projective space, where the vector $u$ is mapped to the line orthogonal to $u$, so that $u$ and its opposite $-u$ are identified. Under this map one has that $\ND(X) \subseteq \pi^{-1}(\NE(X))$. Hence the local topological properties  of $\ND(X)$ and $\NE(X)$ are the same. However, this inclusion can be strict as Ledrappier's example \ref{ex:Ledrappier} shows: $\ND(X)$ consists only of $3$ vectors and none of their opposites are in $\ND(X)$.

One could ask the relative version of \cref{ques:realization}, that is: {\it Given a topological dynamical system $(Y,S,\Z^d)$ and a closed set $N \subseteq \S^{d-1}$ such that the origin  of $\R^d$ belongs to its convex hull, does there exist a topological dynamical system $(X,T, \Z^d)$ which is an infinite extension of $(Y,S,\Z^d)$ and $\ND(X,\pi) = N$?} This question is actually equivalent to \cref{ques:realization}. To see this, assuming that there is a positive answer to \cref{ques:realization}, given $N$ and $(Y,S,\Z^d)$, there exists $(X,T,\Z^d)$ such that $\ND(X)=N$. Consider $\pi \colon X\times Y\to Y$ the projection onto the second coordinate. It is immediate to check that $\ND(X\times Y,\pi)=\ND(X)=N$, which provides a positive answer to the relative version of \cref{ques:realization}.

 Recall that a horoball of any proper right-invariant distance on $\Z^d$ is included in a $\ell_2$ horoball (see the discussion after \cref{lem:NonemptyHoroball}). It follows that the conclusion of \cref{cor:half_space_Z2} is still true for any proper right-invariant distance  on $\Z^d$, namely the intersection of all the non expansive horoballs is empty.  This leads to ask 

\begin{ques} \label{ques:intersection_nonexpansive}
For which group/metrics,   the intersection of all nonexpansive horoballs of any action  on an infinite space is empty?
\end{ques}

The closeness of the set of nondeterministic vectors is important to get \cref{cor:half_space_Z2}.
It is also natural to wonder if \cref{lem:BLclosdness} is still true for metrics other than  $\ell_2$. More generally, is it always the case that the set of nonexpansive horoballs corresponds to a closed subset of horoballs, for any distance in the group? Before giving a partial answer to this question, let us precise the topology we may consider on the set of horoballs. Since the topology on the horofunctions is the compact open topology, a natural topology on the horoballs is then  given by the Hausdorff convergence on any compact set of the group \footnote{This means that $H_n\to H$ as $n\to \infty$ if and only if for all compact set $K\subseteq G$ and $\epsilon>0$ the Hausdorff distance between $\overline{H_n}\cap K$ and $\overline{H}\cap K$ is smaller than $\epsilon$ for all large enough $n$.}. With this topology, the following simple example shows that the set of nonexpansive horoballs is not always closed, even in $\Z^2$. As we shall see, the distance chosen on the group plays an essential role.

\medskip
\begin{example}[A $\Z^2$-subshift with a nonclosed set of  nondeterministic $\ell_1$ horoballs] \label{example:nonexp_noclosed}
Consider an infinite $\Z$-subshift $X$, \ie, a closed subset $ X \subseteq \{0,1\}^\Z$, invariant by the shift map $\sigma$. 
Define the $\Z^2$ system  $(X,T,\Z^2)$  endowed with the $\ell_1$-norm  where $T_{(n,m)}=\sigma^{n-2m}$. 
One checks that the $\ell_1$ horoball  
$\{(x,y):  -x< y< x\} $ is nonexpansive as well as all its translates along the  half line $y=x$,  $y > 0$.
If the set of nonexpansive horoballs for $\ell_1$ were closed, the horoball   $\{(x,y): y <x  \} $ would be a nonexpansive horoball for this subshift. But, as the action is the identity along the line   $\{2y=x\}\cap \Z^2$, it is straightforward to check that this is imposible
\end{example}

This example leads to the following question related to the geometry of groups.
\begin{ques}
Do there exist conditions on the group, the metric and/or the system so that the set of nonexpansive horoballs of any action is a closed set? Are there other interesting topologies on the set of horoballs for which the set of nonexpansive horoballs is closed? 
\end{ques}

\section{Distality in Cantor dynamics}  \label{sec:Cantor_dynamics}

In this section we illustrate further  applications of Robinson Crusoe theorem (\cref{theo:SchwartzmannVersionAbstraite}) in the context of zero dimensional dynamics, \ie, when the phase space $X$ is a zero dimensional space. 

Recall that for a topological dynamical system $(X,T,G)$ a point $x \in X$ is {\em almost periodic} if for every neighborhood $U$ of $x$ the set of return times $N(x,U) = \{g \in G : T_{g}(x) \in U\}$ is {\em syndetic} (\ie, there exists a compact set $K\subseteq G$ such that $KN(x,U) = G$). Equivalently, $x$ has a minimal orbit closure, \ie, $(\overline{Gx}, G)$ is  minimal. The topological dynamical system $(X,T,G)$ is {\em pointwise almost periodic} if every $x\in X$ is almost periodic.

The following result was obtained in \cite{AGW} (stated for finitely generated groups). Here we give a proof for any countable group whose horofunctions satisfy the conditions of \cref{lem:BigHoroball}.

\begin{prop}\label{prop:EquivptwiseRecurrent}
Let $(X,T,G)$ be a dynamical system with $X$ a zero dimensional space and $G$ a countable group. Assume that the group  admits a right-invariant and proper metric so that each horoball $H$ contains arbitrarily large balls. Then the following statements are equivalent,
\begin{enumerate}
\item\label{item:propER1} the action of $G$ is pointwise almost periodic;
\item\label{item:propER2} the relation $R= \{(x,y) \in  X\times X : y \in \overline {Gx}\}$ is closed. 
\end{enumerate}
\end{prop} We notice  that the conclusion of  \cref{prop:EquivptwiseRecurrent} is still true when the acting group $G$ is finite. 
\begin{proof}  The implication \eqref{item:propER2} $\Rightarrow$  \eqref{item:propER1} is valid for any topological $G$-action. By contradiction, assume there exists $x \in X$ that is not almost periodic. Then, there exists an almost periodic point $y \in \overline{Gx}$ such that $x \notin \overline{Gy}$. Let $(g_{n})_{n\in \N} \subseteq G$ be a sequence such that $T_{g_{n}}(x)$ converges to $y$ as $n$ goes to infinity. It follows that  $(T_{g_{n}}(x), x) \in R$ converges to $(y,x) \in R$, a contradiction.

We also prove \eqref{item:propER1} $\Rightarrow$  \eqref{item:propER2} by contradiction. Assume that $R$ is not closed. Then there exists a sequence $((x_{n},y_{n}))_{n\in \N}$ in $R$ converging to $(x,y) \notin R$. Let $U$ be a clopen subset such that $\overline{Gx} \subseteq U$ and $y \notin U$. Let $O$ be the set $\{y' \in \overline{Gx}^{c} : \overline{Gy'} \cap U^{c} \neq \emptyset\}$. One checks that $O$ is a $G$-invariant open set, which is not closed since $x_n$ belongs to $O$ for large enough $n$. Moreover, $U$ is also a neighborhood of $\partial O$. It follows from Theorem \ref{theo:SchwartzmannVersionAbstraite} that there exist $y' \in O$ and a horoball $H$ such that $T_{g}(y') \in U$ for every $g \in H$. Since each horoball contains arbitrarily large balls, the set of return times $N(y", U^{c})$ of some point $y"\in U^c$ in the orbit of $y'$, cannot be syndetic. This is a contradiction.
\end{proof}

It follows from \cref{lem:BigHoroball} that any finitely generated group fulfills   the assumptions of \cref{prop:EquivptwiseRecurrent}. It would be interesting to exhibit other groups or families of groups that satisfy the assumptions of \cref{prop:EquivptwiseRecurrent}. For now, we do not rule out that this family consists only of finitely generated groups.

Let $\pi\colon (X,T, G)\to (Y,S,G)$ be a factor map between the topological dynamical systems $(X,T, G)$ and $(Y,S,G)$. We say that $\pi$ is {\em equicontinuous} (or that $X$ is an equicontinuous extension of $Y$) if for any $\epsilon>0$ there exists $\delta>0$ such that $\dist(x,y)<\delta$ and $\pi(x)=\pi(y)$ implies that $\dist(T_g(x),T_g(y))<\epsilon$ for all $g\in G$. The factor map $\pi$ is {\em distal} (or $X$ is a distal extension of $Y$) if $\displaystyle \inf_{g\in G}{\rm dist} (T_{g}(x), T_{g}(y)) >0$ whenever $\pi(x)=\pi(y)$ and $x \neq y$. 
If the factor $\pi$ is trivial one says that $(X,T,G)$ is equicontinous, or distal, respectively.  

Two points $x,y\in X$ are {\em $\pi$-regionally proximal} if for any $\epsilon>0$ there exist points $x',y'\in X$ and $g\in G$ such that $\pi(x')=\pi(y')$, ${\rm dist} (x,x')<\epsilon$, $
{\rm dist} (y,y')<\epsilon$ and ${\rm dist} (T_g x',T_g y')<\epsilon$. We let $\RP(\pi)$ denote the set of $\pi$-regionally proximal points.  
It is well known (see for instance \cite[Chapter 7, Theorem 2]{Aus}) that a factor map is equicontinuous if and only if $\RP(\pi)=\triangle_X$ and that if a factor map is equicontinuous, then it is a distal factor. 

\begin{cor} \label{cor:factor_distals_are_equicontinuous}
Let $\pi\colon (X,T,G)\to (Y,S,G)$ be a factor map, where $G$ is a countable group, $X$ is zero dimensional and $(Y,S,G)$ is a pointwise almost periodic system. Assume that the group  admits a right-invariant and proper metric so that each horoball $H$ contains arbitrarily large balls.
Then, the factor $\pi$ is distal if and only if it is equicontinuous.
\end{cor}

\begin{proof}
We show the nontrivial implication. Assume that $\pi$ is distal and consider the zero dimensional system $(R_{\pi},T^{(2)}, G)$, where $R_\pi=\{(x,y)\in X\times X: \pi(x)=\pi(y)\}$ and $T^{(2)}$ is the diagonal action $T^{(2)}_g(x,y)= (T_g(x),T_g(y))$. 
We claim that the distality of $\pi$ implies that $(R_{\pi},G)$ is pointwise almost periodic. To show this, we need to use the machinery of the enveloping semigroup (to avoid details about this theory, we refer to the book \cite[Chapters 3 and 6]{Aus}). Let $E(X,T,G)$ and $E(Y,S,G)$ denote the enveloping semigroup of the systems $(X,T,G)$ and $(Y,S,G)$ respectively. Since $\pi(x)$ is an almost periodic point, we can find $u\in E(Y,S,G)$ a minimal idempotent such that $u\pi(x)=\pi(x)$. We can take a minimal idempotent $v\in E(X,T,G)$ such that $\pi_{\ast}(v)=u$, where $\pi_{\ast}\colon E(X,T,G)\to E(Y,S,G)$ is the natural induced semigroup homomorphism (see \cite[Chapter 3, Theorem 7]{Aus}). Then, $v(x,y)=(vx,vy)$ and, since $u\pi(x)=\pi(x)$,  we obtain $\pi(vx)=\pi(x)=\pi(y)=\pi(vy)$. Noting that $(x,vx)$ and $(y,vy)$ are proximal pairs (\ie, $\inf_{g\in G} \dist(x,vx)=\inf_{g\in G} \dist(y,vy)=0$), we obtain that $x=vx$ and $y=vy$. Hence, $(x,y)=v(x,y)$ is a minimal point.

By \cref{prop:EquivptwiseRecurrent}, the orbit relation $R=\{((x_1,y_1),(x_2,y_2)) \in R_{\pi}\times R_{\pi} : (x_2,y_2) \in \overline{G(x_1,y_1)}\}$ is closed. From this, it is classical to deduce that $\RP(\pi)=\triangle_X$, we provide the argument for completeness. Let $(x,y)\in \RP(\pi)$ and let $(x_n,y_n)\in R_{\pi}$ and $g_n\in G$ such that $(x_n,y_n)\to (x,y)$ and $(T_{g_n}x_n,T_{g_n}y_n)\to (z,z)$, for $z\in X$ and a sequence $(g_n)$ in $G$. Since $((T_{g_n}x_n,T_{g_n}y_n),(x_n,y_n))\in R$, and this relation is closed, we obtain that $((z,z),(x,y))\in R$. This implies that $x=y$. 
We conclude that $\pi$ is an equicontinuous extension. 

\end{proof}

When $(Y,S,G)$ is the trivial system, we obtain the following corollary.
\begin{cor}\label{cor:NoDistalCantorAction}
Let $(X,T,G)$ be a topological dynamical where $X$ is zero dimensional and $G$ is countable. Assume that the group $G$ admits a proper and right-invariant metric so that each horoball contains arbitrarily large balls. Then, the action of $G$ is distal if and only if it is equicontinuous.
\end{cor}
We remark that \cref{cor:NoDistalCantorAction} was obtained in \cite{AGW} under the assumption that $G$ is finitely generated.  Moreover,  \cref{cor:NoDistalCantorAction} does not hold for all countable groups. In \cite{MW} (page 266), the authors constructed examples of (non finitely generated) groups that admit distal  but non equicontinuous  actions on a Cantor space. We thank Benjy Weiss and Xiangdong Ye for pointing us to this example and reference.

The next corollary is  Theorem 1.3 in \cite{MS}. Unlike \cite{MS}, we derive it as a consequence of \cref{cor:NoDistalCantorAction}.

For a group $G$, and a subgroup $S\leqslant G$, we let ${\rm Cent}_{G}(S)$ denote the {\em centralizer} of $S$ in $G$, \ie, the set of elements $g \in G$ such that $gs=sg$ for all $s\in S$. It is a group under composition.   
For a topological dynamical system $(X,T,G)$, let ${\rm Homeo}(X)$ denote the group of all the self-homeomorphisms of $X$ and $ \langle T \rangle$ the group generated by the elements $T_g \in {\rm Homeo}(X)$, $g \in G$.

\begin{cor}\cite{MS}\label{cor:MeyerovitchSalo}
Let $(X,T,G)$ be a $G$-subshift for a countable group $G$. Let $N \leqslant {\rm Cent}_{{\rm Homeo}(X)} (\langle T \rangle)$ be a {finitely generated} subgroup. If each $N$-orbit in $X$ is finite, then the group $N$ is finite.
\end{cor}
In particular, when $N$ is reduced to  shift maps, there is no infinite subshift for a  finitely generated group $G$ where all the shift-orbits are finite. 

\begin{proof}
Since the shift action is expansive and $X$ is zero dimensional, it is well known that the centralizer group ${\rm Cent}_{{\rm Homeo}(X)} (\langle T \rangle)$ is discrete in ${\rm Homeo}(X)$ for the uniform topology.
If each $N$-orbit is finite, its  action is distal. It follows from Corollary \ref{cor:NoDistalCantorAction} that the action is equicontinuous, hence $N$ is a compact subgroup of ${\rm Homeo}(X)$.  As a discrete group, it is finite. 
\end{proof}

\end{document}